\documentclass[11pt]{article}
\usepackage{amssymb,amsmath,amsfonts,amsthm,mathrsfs}
\usepackage{graphicx,graphics,psfrag,epsfig,calrsfs,cancel}
\usepackage[colorlinks=true, pdfstartview=FitV, linkcolor=blue, citecolor=red, urlcolor=blue]{hyperref}

\usepackage{caption,subfig,color}
\usepackage{tikz}
\usetikzlibrary{calc}

\usepackage{geometry}
\newtheorem{assumption}{Assumption}
\newtheorem{remark}{Remark}
\newtheorem{lemma}{Lemma}
\newtheorem{proposition}{Proposition}
\newtheorem{theorem}{Theorem}

\newtheorem{definition}{Definition}

\newcommand{\N}{{\mathbb N}}
\newcommand{\Z}{{\mathbb Z}}
\newcommand{\R}{{\mathbb R}}
\newcommand{\C}{{\mathbb C}}
\newcommand{\T}{{\mathbb T}}

\newcommand{\LL}{{\mathbb L}}
\newcommand{\M}{{\mathbb M}}
\newcommand{\D}{{\mathbb D}}
\newcommand{\Dbar}{\overline{\mathbb D}}
\newcommand{\U}{{\mathcal U}}
\newcommand{\Ubar}{\overline{\mathcal U}}
\newcommand{\cercle}{{\mathbb S}^1}

\newcommand{\dps}{\displaystyle}
\newcommand{\Ng}{| \! | \! |}
\newcommand{\Nd}{| \! | \! |}

\begin{document}

\title{The Leray-G{\aa}rding method\\
for finite difference schemes}

\author{Jean-Fran\c{c}ois {\sc Coulombel}\thanks{CNRS and Universit\'e de Nantes, Laboratoire de Math\'ematiques 
Jean Leray (UMR CNRS 6629), 2 rue de la Houssini\`ere, BP 92208, 44322 Nantes Cedex 3, France. Email: 
{\tt jean-francois.coulombel@univ-nantes.fr}. Research of the author was supported by ANR project BoND, 
ANR-13-BS01-0009-01.}}
\date{\today}
\maketitle

\begin{abstract}
In \cite{leray} and \cite{garding}, {\sc Leray} and {\sc G{\aa}rding} have developed a multiplier technique for 
deriving a priori estimates for solutions to scalar hyperbolic equations in either the whole space or the torus. 
In particular, the arguments in \cite{leray,garding} provide with at least one {\it local} multiplier and one 
{\it local} energy functional that is controlled along the evolution. The existence of such a local multiplier 
is the starting point of the argument by {\sc Rauch} in \cite{rauch} for the derivation of semigroup estimates 
for hyperbolic initial boundary value problems. In this article, we explain how this multiplier technique can be 
adapted to the framework of finite difference approximations of transport equations. The technique applies to 
numerical schemes with arbitrarily many time levels, and encompasses a somehow magical trick that has been 
known for a long time for the leap-frog scheme. More importantly, the existence and properties of the local 
multiplier enable us to derive {\it optimal} semigroup estimates for fully discrete hyperbolic initial boundary 
value problems, which answers a problem raised by {\sc Trefethen}, {\sc Kreiss} and {\sc Wu} 
\cite{trefethen3,kreiss-wu}.
\end{abstract}
\bigskip

\noindent {\small {\bf AMS classification:} 65M06, 65M12, 35L03, 35L04.}

\noindent {\small {\bf Keywords:} hyperbolic equations, difference approximations, stability, boundary conditions, 
semigroup.}
\bigskip
\bigskip


Throughout this article, we use the notation
\begin{align*}
&\U := \{\zeta \in \C,|\zeta|>1 \}\, ,\quad \Ubar := \{\zeta \in \C,|\zeta| \ge 1 \}\, ,\\
&\D := \{\zeta \in \C,|\zeta|<1 \}\, ,\quad \cercle := \{\zeta \in \C,|\zeta|=1 \} \, ,\quad \Dbar := \D \cup \cercle \, .
\end{align*}
We let ${\mathcal M}_n ({\mathbb K})$ denote the set of $n \times N$ matrices with entries in ${\mathbb K} 
= \R \text{ or } \C$. If $M \in {\mathcal M}_n (\C)$, $M^*$ denotes the conjugate transpose of $M$. We let $I$ 
denote the identity matrix or the identity operator when it acts on an infinite dimensional space. We use the 
same notation $x^* \, y$ for the Hermitian product of two vectors $x,y \in \C^n$ and for the Euclidean product 
of two vectors $x,y \in \R^n$. The norm of a vector $x \in \C^n$ is $|x| := (x^* \, x)^{1/2}$. The induced matrix 
norm on ${\mathcal M}_n (\C)$ is denoted $\| \cdot \|$.

The letter $C$ denotes a constant that may vary from line to line or within the same line. The dependence of the 
constants on the various parameters is made precise throughout the text.

In what follows, we let $d \ge 1$ denote a fixed integer, which will stand for the dimension of the space domain 
we are considering. We shall also use the space $\ell^2$ of square integrable sequences. Sequences may be 
valued in $\C^k$ for some integer $k$. Some sequences will be indexed by $\Z^{d-1}$ while some will be indexed 
by $\Z^d$ or a subset of $\Z^d$. We thus introduce some specific notation for the norms. Let $\Delta x_i>0$ for 
$i=1,\dots,d$ be $d$ space steps. We shall make use of the $\ell^2(\Z^{d-1})$-norm that we define as follows: for 
all $v \in \ell^2 (\Z^{d-1})$,
\begin{equation*}
\| v \|_{\ell^2(\Z^{d-1})}^2 := \left( \prod_{k=2}^d \Delta x_k \right) \, \sum_{i=2}^d \sum_{j_i\in \Z} |v_{j_2,\dots,j_d}|^2 \, .
\end{equation*}
The corresponding scalar product is denoted $\langle \cdot,\cdot \rangle_{\ell^2(\Z^{d-1})}$. Then for all integers 
$m_1 \le m_2$, we set
\begin{equation*}
\Ng u \Nd_{m_1,m_2}^2 := \Delta x_1 \, \sum_{j_1=m_1}^{m_2} \|u_{j_1,\cdot} \|_{\ell^2(\Z^{d-1})}^2 \, ,
\end{equation*}
to denote the $\ell^2$-norm on the set $[m_1,m_2] \times \Z^{d-1}$ ($m_1$ may equal $-\infty$ and $m_2$ 
may equal $+\infty$). The corresponding scalar product is denoted $\langle \cdot,\cdot \rangle_{m_1,m_2}$. 
Other notation throughout the text is meant to be self-explanatory.

\section{Introduction}
\label{intro}

\subsection{Some motivations and a brief reminder}

The ultimate goal of this article is to derive semigroup estimates for finite difference approximations of 
hyperbolic initial boundary value problems. Up to now, the only available general stability theory for such 
numerical schemes is due to {\sc Gustafsson}, {\sc Kreiss} and {\sc Sundstr{\"o}m} \cite{gks}. It relies on 
a Laplace transform with respect to the time variable, and the corresponding stability estimates are thereby 
restricted to zero initial data. A long standing problem in this line of research is, starting from the GKS stability 
estimates, which are {\it resolvent} type estimates, to incorporate nonzero initial data and to derive {\it semigroup} 
estimates, see, e.g., the discussion in \cite[section 4]{trefethen3}. This problem is delicate for the following 
reason: the validity of the GKS stability estimate is known to be equivalent to a {\it slightly stronger version} 
of the resolvent estimate
\begin{equation}
\label{resolventkreiss}
\sup_{z \in \U} \, (|z|-1) \, \| (z\, I-T)^{-1} \|_{{\mathcal L}(\ell^2(\N))} <+\infty \, ,
\end{equation}
where $T$ is some bounded operator on $\ell^2(\N)$ that incorporates both the discretization of the hyperbolic 
equation and the numerical boundary conditions. Deriving an optimal semigroup estimate amounts to showing 
that $T$ is power bounded. In finite dimension, the equivalence between power boundedness of $T$ and the 
resolvent condition \eqref{resolventkreiss} is known as the {\sc Kreiss} matrix Theorem, but the analogous 
equivalence is known to fail in general in infinite dimension. Worse, even the strong resolvent condition
\begin{equation*}
\sup_{n \ge 1} \, \sup_{z \in \U} \, (|z|-1)^n \, \| (z\, I-T)^{-n} \|_{{\mathcal L}(\ell^2(\N))} <+\infty \, ,
\end{equation*}
does not imply in general that $T$ is power bounded, see, e.g., the review \cite{strikwerda-wade} or \cite{TE} 
for details and historical comments.

Optimal semigroup estimates have nevertheless been derived for some discretized hyperbolic initial boundary 
value problems. More specifically, the first general derivation of semigroup estimates is due to {\sc Wu} \cite{wu}, 
whose analysis deals with numerical schemes with two time levels and scalar equations. The results in 
\cite{wu} were extended by {\sc Gloria} and the author in \cite{jfcag} to systems in arbitrary space dimension, 
but the arguments in \cite{jfcag} are still restricted to numerical schemes with two time levels. The present article 
gives, as far as we are aware of, the first systematic derivation of semigroup estimates for fully discrete hyperbolic 
initial boundary value problems in the case of numerical schemes with arbitrarily many time levels. It generalizes 
the arguments of \cite{wu,jfcag} and provides new insight for the construction of \lq \lq dissipative\rq \rq $\,$ 
numerical boundary conditions for discretized evolution equations. Let us observe that the leap-frog scheme, 
with some specific boundary conditions, has been dealt with by {\sc Thomas} \cite{thomas} by using a 
{\it multiplier} technique. It is precisely this technique which we aim at developing in a systematic fashion 
for numerical schemes with arbitrarily many time levels. In particular, we shall explain why the somehow 
magical multiplier $u_j^{n+2}+u_j^n$ for the leap-frog scheme, see, e.g., \cite{RM}, follows from a general 
theory that is the analogue of the {\sc Leray}-{\sc G{\aa}rding} method for partial differential equations, 
which we briefly recall now.
\bigskip

The method by {\sc Leray} and {\sc G{\aa}rding} \cite{leray,garding} provides with suitable multipliers 
for scalar hyperbolic operators of arbitrary order. Namely, given an integer $m \ge 0$, we consider a 
partial differential operator of the form
\begin{equation*}
L :=\partial_t^{m+1} +\sum_{k=1}^{m+1} P_k(\partial_x) \, \partial_t^{m+1-k} \, ,
\end{equation*}
where $t \in \R$ stands for the time variable, $x \in \R^d$ stands for the space variable\footnote{The periodic 
case $x \in \T^d$ can be dealt with in a similar way and is actually the one considered in \cite{garding}.}, and 
each operator $P_k(\partial_x)$ is a linear combination of spatial partial derivatives of order $k$:
\begin{equation*}
P_k(\partial_x) =\sum_{|\alpha|=k} p_{k,\alpha} \, \partial_x^\alpha \, ,\quad 
\partial_x^\alpha := \partial_{x_1}^{\alpha_1} \cdots \partial_{x_d}^{\alpha_d} \, ,\quad 
|\alpha| :=\alpha_1 +\cdots +\alpha_d \, .
\end{equation*}
In the above formula, the $p_{k,\alpha}$'s are real numbers\footnote{We restrict here for simplicity to 
linear operators with constant coefficients.}. Well-posedness of the Cauchy problem
\begin{equation}
\label{cauchyedp}
L \, u =0 \, ,\quad (u,\partial_t u,\dots,\partial_t^m u)|_{t=0} =(u_0,u_1,\dots,u_m) \, ,
\end{equation}
in Sobolev spaces is known to be linked with {\it hyperbolicity} of $L$. Namely, if $L$ is strictly hyperbolic, 
meaning that for all $\xi \in \R^d \setminus \{ 0\}$, the (homogeneous) polynomial
\begin{equation}
\label{symboledp}
P(\tau,\xi) :=\tau^{m+1} +\sum_{k=1}^{m+1} P_k(i \, \xi) \, \tau^{m+1-k} \, ,\quad 
P_k(i \, \xi) :=i^k \, \sum_{|\alpha|=k} p_{k,\alpha} \, \xi^\alpha \, ,
\end{equation}
has $m+1$ simple purely imaginary roots with respect to $\tau$, then the Cauchy problem \eqref{cauchyedp} 
is well-posed in $H^m(\R^d) \times \cdots \times L^2(\R^d)$. In particular, there exists a constant $C>0$, that 
is independent of the solution $u$ and the initial data $u_0,u_1,\dots,u_m$, such that there holds:
\begin{equation}
\label{estimedp}
\sup_{t \in \R} \, \, \sum_{k=0}^m \| \partial_t^k u(t) \|_{H^{m-k}(\R^d)} \le C \, 
\sum_{k=0}^m \| u_k \|_{H^{m-k}(\R^d)} \, .
\end{equation}

The method by {\sc Leray} and {\sc G{\aa}rding} gives a quick and elegant way to derive the estimate 
\eqref{estimedp} assuming that the solution $u$ to \eqref{cauchyedp} is sufficiently smooth. By standard 
duality arguments, the validity of the a priori estimate \eqref{estimedp} yields well-posedness -meaning 
existence, uniqueness and continuous dependence on the data- for \eqref{cauchyedp}. Hence the main 
point is to prove \eqref{estimedp} assuming that $u$ is sufficiently smooth and decaying at infinity so that 
all integration by parts arising in the computations are legitimate. The main idea is to find a suitable quantity 
$M \, u$, which we call a multiplier and that will be linear with respect to $u$, such that when integrating the 
quantity $0=(M \, u) \, (L \, u)$ on the slab $[0,T] \times \R^d$, one gets the estimate \eqref{estimedp} {\it for 
free} (negative times are obtained by changing $t \rightarrow -t$). Following \cite[Chapter VI]{leray} and 
\cite[Section 3]{garding}, one possible choice of a multiplier is given by $L' \, u$ where $L'$ stands for the 
partial differential operator of order $m$ whose symbol is $\partial_\tau P$, with $P$ given in \eqref{symboledp}. 
Why $L' \, u$ is a good multiplier is justified in \cite{leray,garding}. A well-known particular case is the choice 
of $2\, \partial_t u$ as a multiplier for the wave equation. Here $P(\tau,\xi)=\tau^2 +|\xi|^2$ and therefore 
$\partial_\tau P =2\, \tau$, hence the choice $2\, \partial_t u$. The latter quantity is indeed a suitable multiplier 
for the wave operator because of the formula\footnote{We refer to \cite[page 74]{garding} for the generalization 
of such "integration by parts" formula to partial derivatives of higher order.}:
\begin{equation*}
2\, \partial_t u \, (\partial_t^2 u -\Delta_x u) =\partial_t \Big( (\partial_t u)^2 +\sum_{j=1}^d (\partial_{x_j} u)^2 
\Big) -2 \, \text{\rm div}_x \big( \partial_t u \, \nabla_x u \big) \, .
\end{equation*}
The important fact here is that the {\it energy}:
\begin{equation*}
(\partial_t u)^2 +\sum_{j=1}^d (\partial_{x_j} u)^2 \, ,
\end{equation*}
is a positive definite quadratic form of the first order partial derivatives of $u$. Let us observe that the multiplier 
$L' \, u$ is {\it local}, meaning that its pointwise value at $(t,x)$ only depends on $u$ in a neighborhood of $(t,x)$. 
This is important in view of using this multiplier in the study of initial boundary value problems. Another important 
remark is that the above energy is also {\it local}, and the arguments in \cite{leray,garding} show that this property 
is not specific to the wave operator. The fact that both the multiplier and the energy are local is crucial in the 
arguments of \cite[Lemma 1]{rauch}. In our framework of discretized equations, the multiplier will be local but 
the energy will not necessarily be so. We shall not exactly follow the arguments of \cite{rauch} which use time 
reversibility, but rather construct dissipative boundary conditions which will yield the optimal semigroup estimate 
we are aiming at.

\subsection{The main result}

We first set a few notations. We let $\Delta x_1,\dots,\Delta x_d,\Delta t>0$ denote space and time steps 
where the ratios, the so-called {\sc Courant-Friedrichs-Lewy} parameters, $\lambda_i :=\Delta t/\Delta x_i$, 
$i=1,\dots,d$, are fixed positive constants. We keep $\Delta t \in (0,1]$ as a small parameter and let the 
space steps $\Delta x_1,\dots,\Delta x_d$ vary accordingly. The $\ell^2$-norms with respect to the space 
variables have been previously defined and thus depend on $\Delta t$ and the CFL parameters through 
the mesh volume ($\Delta x_2 \cdots \Delta x_d$ on $\Z^{d-1}$, and $\Delta x_1 \cdots \Delta x_d$ on 
$\Z^d$). We always identify a sequence $w$ indexed by either $\N$ (for time), $\Z^{d-1}$ or $\Z^d$ (for 
space), with the corresponding step function. In particular, we shall feel free to take Fourier or Laplace 
transforms of such sequences.

For all $j\in \Z^d$, we set $j=(j_1,j')$ with $j':=(j_2,\dots,j_d) \in \Z^{d-1}$. We let $p,q,r \in \N^d$ denote some 
fixed multi-integers, and define $p_1,q_1,r_1$, $p',q',r'$ according to the above notation. We also let $s \in \N$ 
denote some fixed integer. We consider a recurrence relation of the form:
\begin{equation}
\label{numibvp}
\begin{cases}
{\dps \sum_{\sigma=0}^{s+1}} Q_\sigma \, u_j^{n+\sigma} =\Delta t \, F_j^{n+s+1} \, ,& 
j \in \Z^d \, ,\quad j_1 \ge 1\, ,\quad n\ge 0 \, ,\\
u_j^{n+s+1} +{\dps \sum_{\sigma=0}^{s+1}} B_{j_1,\sigma} \, u_{1,j'}^{n+\sigma} =g_j^{n+s+1} \, ,& 
j \in \Z^d \, ,\quad j_1=1-r_1,\dots,0\, ,\quad n\ge 0 \, ,\\
u_j^n = f_j^n \, ,& j \in \Z^d \, ,\quad j_1\ge 1-r_1\, ,\quad n=0,\dots,s \, ,
\end{cases}
\end{equation}
where the operators $Q_\sigma$ and $B_{j_1,\sigma}$ are given by:
\begin{equation}
\label{defop}
Q_\sigma := \sum_{\ell_1=-r_1}^{p_1} \sum_{\ell'=-r'}^{p'} a_{\ell,\sigma} \, {\bf S}^\ell \, ,\quad 
B_{j_1,\sigma} := \sum_{\ell_1=0}^{q_1} \sum_{\ell'=-q'}^{q'} b_{\ell,j_1,\sigma} \, {\bf S}^\ell \, .
\end{equation}
In \eqref{defop}, the $a_{\ell,\sigma},b_{\ell,j_1,\sigma}$ are {\it real numbers} and are independent of the 
small parameter $\Delta t$ (they may depend on the CFL parameters though), while ${\bf S}$ denotes the 
shift operator on the space grid: $({\bf S}^\ell v)_j :=v_{j+\ell}$ for $j,\ell \in \Z^d$. We have also used the 
short notation
\begin{equation*}
\sum_{\ell'=-r'}^{p'} :=\sum_{i=2}^d \sum_{\ell_i=-r_i}^{p_i} \, ,\quad 
\sum_{\ell'=-q'}^{q'} :=\sum_{i=2}^d \sum_{\ell_i=-q_i}^{q_i} \, .
\end{equation*}

The numerical scheme \eqref{numibvp} is understood as follows: one starts with $\ell^2$ initial data $(f_j^0)$, ..., 
$(f_j^s)$ defined for $j_1 \ge 1-r_1$. Assuming that the solution has been defined up to some time index $n+s$, 
$n \ge 0$, then the first and second equations in \eqref{numibvp} should uniquely determine $u_j^{n+s+1}$ for 
$j_1 \ge 1-r_1$, $j' \in \Z^{d-1}$. The meshes associated with $j_1 \ge 1$ correspond to the {\it interior domain} 
while those associated with $j_1=1-r_1,\dots,0$ represent the {\it discrete boundary}. We wish to deal here 
simultaneously with explicit and implicit schemes and therefore make the following solvability assumption.

\begin{assumption}[Solvability of \eqref{numibvp}]
\label{assumption0}
The operator $Q_{s+1}$ is an isomorphism on $\ell^2 (\Z^d)$. Moreover, for all $(F_j) \in \ell^2 (\N^* \times \Z^{d-1})$ 
and for all $g_{1-r_1,\cdot},\dots,g_{0,\cdot} \in \ell^2 (\Z^{d-1})$, there exists a unique solution $(u_j)_{j_1 \ge 1-r_1} 
\in \ell^2$ to the equations
\begin{equation*}
\begin{cases}
Q_{s+1} \, u_j =F_j \, ,& j \in \Z^d \, ,\quad j_1 \ge 1\, ,\\
u_j +B_{j_1,s+1} \, u_{1,j'} =g_j \, ,& j \in \Z^d \, ,\quad j_1=1-r_1,\dots,0\, .
\end{cases}
\end{equation*}
\end{assumption}

\noindent In particular, Assumption \ref{assumption0} is trivially satisfied in the case of explicit schemes for which 
$Q_{s+1}$ is the identity ($a_{\ell,s+1} =\delta_{\ell_1,0} \cdots \delta_{\ell_d,0}$ in \eqref{defop}, with $\delta$ the 
Kronecker symbol).

The first and second equations in \eqref{numibvp} therefore uniquely determine $u_j^{n+s+1}$ for $j_1 \ge 1-r_1$, 
and one then proceeds to the following time index $n+s+2$. Existence and uniqueness of a solution $(u_j^n)$ 
to \eqref{numibvp} follows from Assumption \ref{assumption0}, so the last requirement for well-posedness is 
continuous dependence of the solution on the three possible source terms $(F_j^n)$, $(g_j^n)$, $(f_j^n)$. 
This is a {\it stability} problem for which several definitions can be chosen according to the functional framework. 
The following one dates back to \cite{gks} in one space dimension and was also considered by {\sc Michelson} 
\cite{michelson} in several space dimensions. It is specifically relevant when the boundary conditions are 
non-homogeneous ($(g_j^n) \not \equiv 0$):

\begin{definition}[Strong stability]
\label{defstab1}
The finite difference approximation \eqref{numibvp} is said to be "strongly stable" if there exists a constant 
$C$ such that for all $\gamma>0$ and all $\Delta t \in \, (0,1]$, the solution $(u_j^n)$ to \eqref{numibvp} 
with $(f_j^0) =\dots =(f_j^s) =0$ satisfies the estimate:
\begin{multline}
\label{stabilitenumibvp}
\dfrac{\gamma}{\gamma \, \Delta t+1} \, \sum_{n\ge s+1} \Delta t \, {\rm e}^{-2\, \gamma \, n\, \Delta t} \, 
\Ng u^n \Nd_{1-r_1,+\infty}^2 +\sum_{n\ge s+1} \Delta t \, {\rm e}^{-2\, \gamma \, n\, \Delta t} \, 
\sum_{j=1-r_1}^{p_1} \| u_{j_1,\cdot}^n \|_{\ell^2 (\Z^{d-1})}^2 \\
\le C \left\{ \dfrac{\gamma \, \Delta t+1}{\gamma} \, \sum_{n\ge s+1} \, \Delta t \, 
{\rm e}^{-2\, \gamma \, n\, \Delta t} \, \Ng F^n \Nd_{1,+\infty}^2 
+\sum_{n\ge s+1} \Delta t \, {\rm e}^{-2\, \gamma \, n\, \Delta t} \, \sum_{j_1=1-r_1}^0 
\| g_{j_1,\cdot}^n \|_{\ell^2 (\Z^{d-1})}^2 \right\} \, .
\end{multline}
\end{definition}

The main contributions in \cite{gks,michelson} are to show that strong stability can be characterized by a 
certain {\it algebraic condition}, which is usually referred to as the Uniform {\sc Kreiss-Lopatinskii} Condition, 
see \cite{jfcnotes} for an overview of such results. We do not pursue such arguments here but rather assume 
from the start that \eqref{numibvp} is strongly stable. We can thus control, with zero initial data, $\ell^2$ type 
norms of the solution to \eqref{numibvp}. Our goal is to understand which kind of stability estimate holds for 
the solution to \eqref{numibvp} when one now considers nonzero initial data $(f_j^0),\dots,(f_j^s)$ in $\ell^2$. 
Our main assumption is the following.

\begin{assumption}[Stability for the discrete Cauchy problem]
\label{assumption1}
For all $\xi \in \R^d$, the dispersion relation
\begin{equation}
\label{dispersion}
\sum_{\sigma=0}^{s+1} \widehat{Q_\sigma} ({\rm e}^{i\, \xi_1},\dots,{\rm e}^{i\, \xi_d}) \, z^\sigma =0 \, ,\quad 
\widehat{Q_\sigma}(\kappa) := \sum_{\ell=-r}^p \kappa^\ell \, a_{\ell,\sigma} \, ,
\end{equation}
has $s+1$ simple roots in $\Dbar$. (The von Neumann condition is said to hold when the roots are located 
in $\Dbar$.) In \eqref{dispersion}, we have used the classical notation
\begin{equation*}
\kappa^\ell :=\kappa_1^{\ell_1} \cdots \kappa_d^{\ell_d} \, ,
\end{equation*}
for $\kappa \in (\C \setminus \{ 0 \})^d$ and $\ell \in \Z^d$.
\end{assumption}

\noindent From Assumption \ref{assumption0}, we know that $Q_{s+1}$ is an isomorphism on $\ell^2$, which 
implies by Fourier analysis that $\widehat{Q_{s+1}} ({\rm e}^{i\, \xi_1},\dots,{\rm e}^{i\, \xi_d})$ does not vanish 
for any $\xi \in \R^d$. In particular, the dispersion relation \eqref{dispersion} is a polynomial equation of degree 
$s+1$ in $z$ for any $\xi \in \R^d$. We now make the following assumption, which already appeared in 
\cite{gks,michelson} and several other works on the same topic.

\begin{assumption}[Noncharacteristic discrete boundary]
\label{assumption2}
For $\ell_1=-r_1,\dots,p_1$, $z \in \C$ and $\eta \in \R^{d-1}$, let us define
\begin{equation}
\label{defA-d}
a_{\ell_1}(z,\eta) :=\sum_{\sigma=0}^{s+1} z^\sigma \, \sum_{\ell'=-r'}^{p'} 
a_{\ell,\sigma} \, {\rm e}^{i \, \ell' \cdot \eta} \, .
\end{equation}
Then $a_{-r_1}$ and $a_{p_1}$ do not vanish on $\Ubar \times \R^{d-1}$, and they have nonzero degree 
with respect to $z$ for all $\eta \in \R^{d-1}$.
\end{assumption}

\noindent Our main result is comparable with \cite[Theorem 3.3]{wu} and \cite[Theorems 2.4 and 3.5]{jfcag} 
and shows that strong stability (or "GKS stability") is a sufficient condition for incorporating $\ell^2$ initial 
conditions in \eqref{numibvp} and proving {\it optimal} semigroup estimates. The main price to pay in 
Assumption \ref{assumption1} is that the roots of the dispersion relation \eqref{dispersion}, which are 
nothing but the eigenvalues of the so-called {\it amplification matrix} for the Cauchy problem, need to be 
{\it simple}. This property is satisfied for instance by the leap-frog and modified leap-frog schemes in 
several space dimensions, under an appropriate CFL condition, see Paragraph \ref{examples}. Our 
main result reads as follows.

\begin{theorem}
\label{mainthm}
Let Assumptions \ref{assumption0}, \ref{assumption1} and \ref{assumption2} be satisfied, and assume 
that the scheme \eqref{numibvp} is strongly stable in the sense of Definition \ref{defstab1}. Then there exists 
a constant $C$ such that for all $\gamma>0$ and all $\Delta t \in \, (0,1]$, the solution to \eqref{numibvp} 
satisfies the estimate:
\begin{multline}
\label{estim1d}
\sup_{n \ge 0} \, {\rm e}^{-2\, \gamma \, n\, \Delta t} \, \Ng u^n \Nd_{1-r_1,+\infty}^2 
+\dfrac{\gamma}{\gamma \, \Delta t+1} \, 
\sum_{n\ge 0} \Delta t \, {\rm e}^{-2\, \gamma \, n\, \Delta t} \, \Ng u^n \Nd_{1-r_1,+\infty}^2 \\
+\sum_{n\ge 0} \Delta t \, {\rm e}^{-2\, \gamma \, n\, \Delta t} \, \sum_{j_1=1-r_1}^{p_1} 
\| u_{j_1,\cdot}^n \|_{\ell^2(\Z^{d-1})}^2 \le C \, \left\{ \sum_{\sigma=0}^s \Ng f^\sigma \Nd_{1-r_1,+\infty}^2 
+\dfrac{\gamma \, \Delta t+1}{\gamma} \, 
\sum_{n\ge s+1} \Delta t \, {\rm e}^{-2\, \gamma \, n\, \Delta t} \, \Ng F^n \Nd_{1,+\infty}^2 \right. \\
\left. +\sum_{n\ge s+1} \Delta t \, {\rm e}^{-2\, \gamma \, n\, \Delta t} \, \sum_{j_1=1-r_1}^0 
\| g_{j_1,\cdot}^n \|_{\ell^2(\Z^{d-1})}^2 \right\} \, .
\end{multline}
In particular, the scheme \eqref{numibvp} is "semigroup stable" in the sense that there exists a constant 
$C$ such that for all $\Delta t \in \, (0,1]$, the solution $(u_j^n)$ to \eqref{numibvp} with $(F_j^n) =(g_j^n) 
=0$ satisfies the estimate
\begin{equation}
\label{estimsemigroup}
\sup_{n\ge 0} \, \Ng u^n \Nd_{1-r_1,+\infty}^2 \le C \, \sum_{\sigma=0}^s \Ng f^\sigma \Nd_{1-r_1,+\infty}^2 \, .
\end{equation}
The scheme \eqref{numibvp} is also $\ell^2$-stable with respect to boundary data, see 
\cite[Definition 4.5]{trefethen3}, in the sense that there exists a constant $C$ such that for all $\Delta t \in \, 
(0,1]$, the solution $(u_j^n)$ to \eqref{numibvp} with $(F_j^n)=(f_j^n)=0$ satisfies the estimate
\begin{equation*}
\sup_{n\ge 0} \, \Ng u^n \Nd_{1-r_1,+\infty}^2 \le C \, \sum_{n\ge s+1} \Delta t \, \sum_{j_1=1-r_1}^0 
\| g_{j_1,\cdot}^n \|_{\ell^2(\Z^{d-1})}^2 \, .
\end{equation*}
\end{theorem}

\noindent Theorem \ref{mainthm} gives the optimal semigroup estimate \eqref{estimsemigroup}, and is 
therefore an improvement with respect to our earlier work \cite{jfc} where in one space dimension, and 
under an appropriate {\it non-glancing} condition\footnote{The non-glancing condition is unfortunately not 
met by the leap-frog scheme.}, we were able to derive the estimate (here $r_1=r$, $p_1=p$ since $d=1$):
\begin{multline*}
\dfrac{\gamma}{\gamma \, \Delta t+1} \, 
\sum_{n\ge 0} \Delta t \, {\rm e}^{-2\, \gamma \, n\, \Delta t} \, \Ng u^n \Nd_{1-r,+\infty}^2 
+\sum_{n\ge 0} \Delta t \, {\rm e}^{-2\, \gamma \, n\, \Delta t} \, \sum_{j=1-r}^p |u_j^n|^2 \\
\le C \, \left\{ \sum_{\sigma=0}^s \Ng f^\sigma \Nd_{1-r,+\infty}^2 +\dfrac{\gamma \, \Delta t+1}{\gamma} \, 
\sum_{n\ge s+1} \Delta t \, {\rm e}^{-2\, \gamma \, n\, \Delta t} \, \Ng F^n \Nd_{1,+\infty}^2 
+\sum_{n\ge s+1} \Delta t \, {\rm e}^{-2\, \gamma \, n\, \Delta t} \, \sum_{j=1-r}^0 |g_j^n|^2 \right\} \, .
\end{multline*}
The latter estimate does not incorporate on the left hand side the quantity:
\begin{equation*}
\sup_{n \ge 0} \, {\rm e}^{-2\, \gamma \, n\, \Delta t} \, \Ng u^n \Nd_{1-r,+\infty}^2 \, ,
\end{equation*}
and was unfortunately still not sufficient for deriving the semigroup estimate \eqref{estimsemigroup}. 
Our main contribution in this article is to exhibit a suitable multiplier for the multistep recurrence relation 
in \eqref{numibvp}. With this multiplier, we can readily show that, for zero initial data, the (discrete) 
derivative of an {\it energy} can be controlled, as in \cite{rauch}, by the trace estimate of $(u_j^n)$ 
and this is where strong stability comes into play. This first argument gives Theorem \ref{mainthm} for 
zero initial data (and even for nonzero initial data if the non-glancing condition of \cite{jfc} is satisfied). 
By linearity we can then reduce to the case of zero forcing terms in the interior and on the boundary. 
The next arguments in \cite{rauch} use time reversibility, which basically always fails for numerical 
schemes\footnote{With the notable exception of the leap-frog scheme that is indeed time reversible !}. 
Hence we must find another argument for dealing with nonzero initial data. Hopefully, the properties of 
our multiplier enable us to construct an auxiliary problem, where we modify the boundary conditions of 
\eqref{numibvp}, and for which we can prove optimal semigroup and trace estimates by "hand-made" 
calculations. In other words, we exhibit an alternative set of boundary conditions that yields {\it strict 
dissipativity}. Using these auxiliary numerical boundary conditions, the proof of Theorem \ref{mainthm} 
follows from a standard superposition argument, see, e.g., \cite[Section 4.5]{benzoni-serre} for partial 
differential equations or \cite{wu,jfcag} for numerical schemes.

\begin{remark}
Assumption \ref{assumption2} excludes the case of explicit two level schemes for which $s=0$ and 
$Q_1=I$, for in that case $a_{-r_1}$ and/or $a_{p_1}$ do not depend on $z$. However, this case 
has already been dealt with in \cite{wu,jfcag}, and we shall see in Section \ref{section3} where the 
assumption that $a_{-r_1}$ and $a_{p_1}$ are not constant is involved, and why the proof is actually 
simpler in the case $s=0$ and $Q_1=I$.
\end{remark}

\subsection{Examples}
\label{examples}

\subsubsection{One space dimension}

Our goal is to approximate the outgoing transport equation ($d=1$ here):
\begin{equation}
\label{transport}
\partial_t u +a \, \partial_x u=0 \, ,\quad u|_{t=0} =u_0 \, ,
\end{equation}
with $t,x>0$ and $a<0$. The latter transport equation does not require any boundary condition at $x=0$. 
However, discretizing \eqref{transport} usually requires prescribing numerical boundary conditions, unless 
one considers an upwind type scheme with a space stencil "on the right" (meaning $r_1=0$ in \eqref{numibvp}). 
We now detail two possible multistep schemes for discretizing \eqref{transport}. Both are obtained by the 
so-called method of lines, which amounts to first discretizing the space derivative $\partial_x u$ and then 
choosing an integration technique for discretizing the time evolution, see \cite{gko}.

\paragraph{The leap-frog scheme.}  It is obtained by approximating the space derivative $\partial_x u$ 
by the centered difference $(u_{j+1}-u_{j-1})/(2\, \Delta x)$, and by then applying the so-called Nystr\"om 
method of order $2$, see \cite[Chapter III.1]{hnw}. The resulting approximation reads
\begin{equation*}
u_j^{n+2} +\lambda \, a \, (u_{j+1}^{n+1}-u_{j-1}^{n+1}) -u_j^n =0 \, ,
\end{equation*}
which corresponds to $s=p=r=1$. Recall that $\lambda>0$ denotes the fixed ratio $\Delta t/\Delta x$. Even 
though \eqref{transport} does not require any boundary condition at $x=0$, the leap-frog scheme stencil 
includes one point to the left, and we therefore need to prescribe some numerical boundary condition at $j=0$. 
One possibility\footnote{This is of course not the only possibility and we refer to \cite{gko,oliger,sloan,trefethen3} 
for some other possible choices which might be more meaningful from a consistency  and accuracy point of 
view. Our main concern here is a discussion on stability for \eqref{numibvp} and the Dirichlet boundary conditions 
are a good illustration for this aspect.} is to prescribe the homogeneous or inhomogeneous Dirichlet boundary 
condition. With general source terms, the corresponding scheme reads
\begin{equation}
\label{leapfrog}
\begin{cases}
u_j^{n+2} +\lambda \, a \, (u_{j+1}^{n+1}-u_{j-1}^{n+1}) -u_j^n =\Delta t \, F_j^{n+2} \, ,& 
j\ge 1\, ,\quad n\ge 0 \, ,\\
u_0^{n+2} = g_0^{n+2} \, ,& n\ge 0 \, ,\\
(u_j^0,u_j^1) = (f_j^0,f_j^1) \, ,& j\ge 0 \, .
\end{cases}
\end{equation}
Assumption \ref{assumption0} is trivially satisfied because \eqref{leapfrog} is explicit. The leap-frog scheme satisfies 
Assumption \ref{assumption1} provided that $\lambda \, |a|<1$. In that case, the two roots to the dispersion relation
\begin{equation*}
z^2 +2\, i \, \lambda \, a \, \sin \xi \, z -1 =0 \, ,
\end{equation*}
are simple and have modulus $1$ for all $\xi \in \R$. Assumption \ref{assumption2} is satisfied as long as 
the velocity $a$ is nonzero, for in that case $a_1(z)=-a_{-1}(z)=\lambda \, a \, z$. The scheme \eqref{leapfrog} 
is known to be strongly stable, see \cite{goldberg-tadmor}. In particular, Theorem \ref{mainthm} shows that 
\eqref{leapfrog} is semigroup stable. An illustration of this stability property is given in the numerical simulation 
of a bump function, propagating at speed $a=-1$ towards the left. Homogeneous Dirichlet boundary conditions 
are enforced at $j=0$. The reflection of the bump generates a highly oscillatory wave packet that propagates 
with velocity $+1$ towards the right. The envelope of this wave packet coincides with the profile of the initial 
condition, which indicates that the $\ell^2$-norm is roughly preserved by the evolution. This numerical observation 
is in agreement with semigroup boundedness.

Other choices of numerical boundary conditions for the leap-frog scheme or its fourth order extension are discussed, 
e.g., in \cite{oliger,sloan,thomas,trefethen3}. The main discussion in \cite{oliger,sloan,trefethen3} is to verify strong 
stability for a wide choice of numerical boundary conditions, and if strong stability holds, then Theorem \ref{mainthm} 
automatically gives semigroup boundedness, which was not achieved in these earlier works.

\begin{figure}[h!]
\begin{center}
\includegraphics[trim=20 20 20 20,clip,scale=0.35]{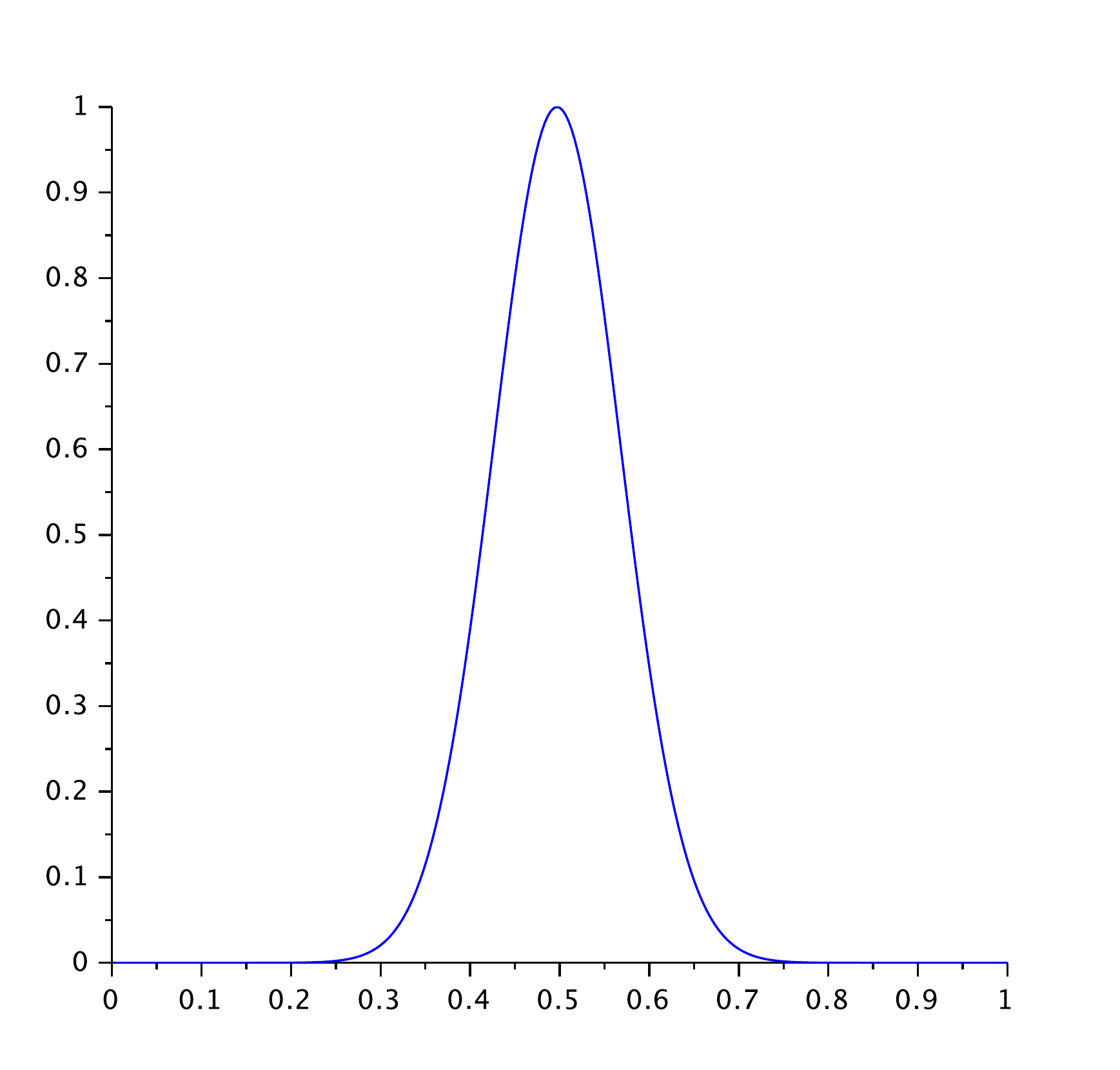}
\includegraphics[trim=20 20 20 20,clip,scale=0.35]{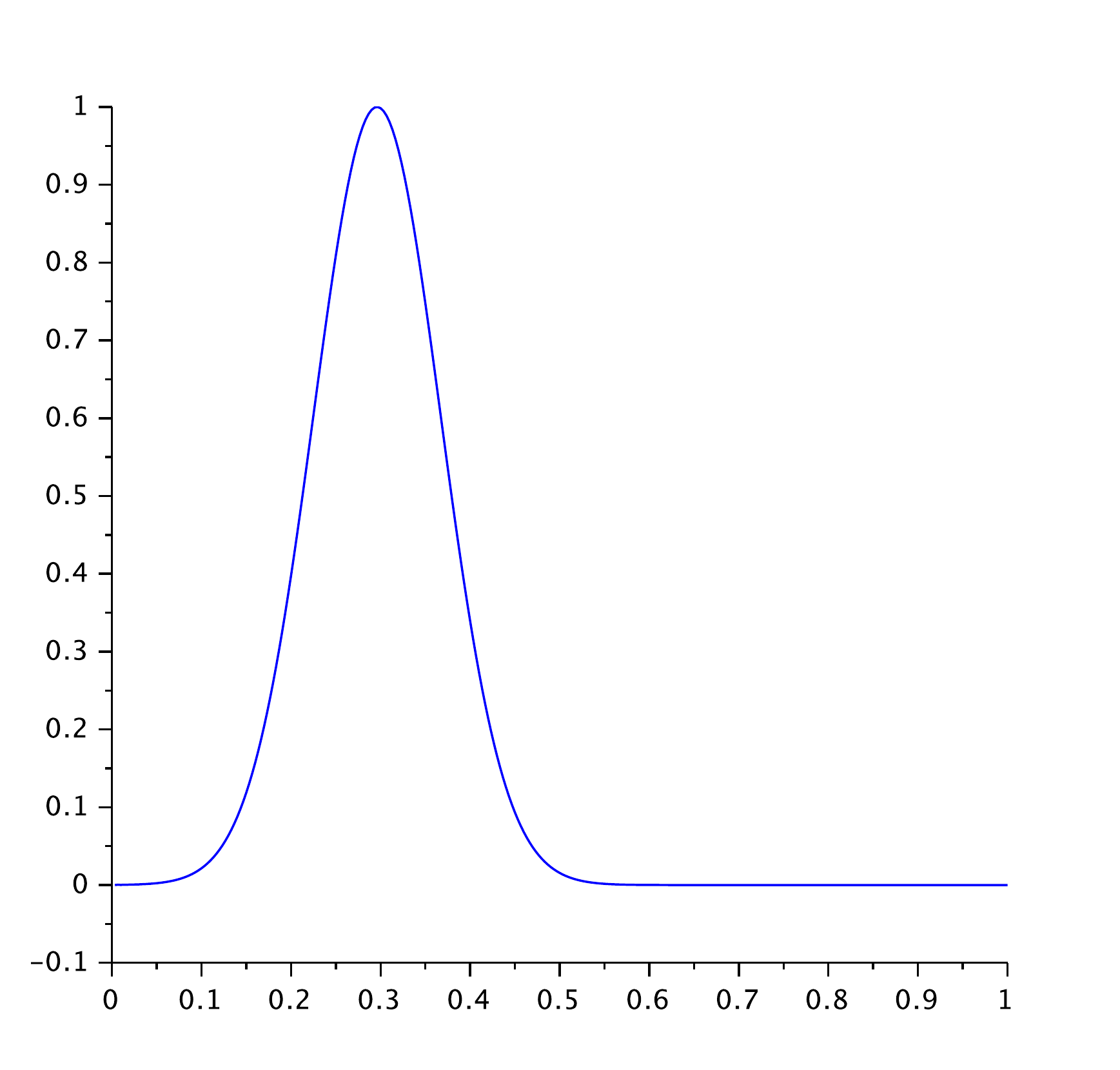}

\includegraphics[trim=20 20 20 20,clip,scale=0.35]{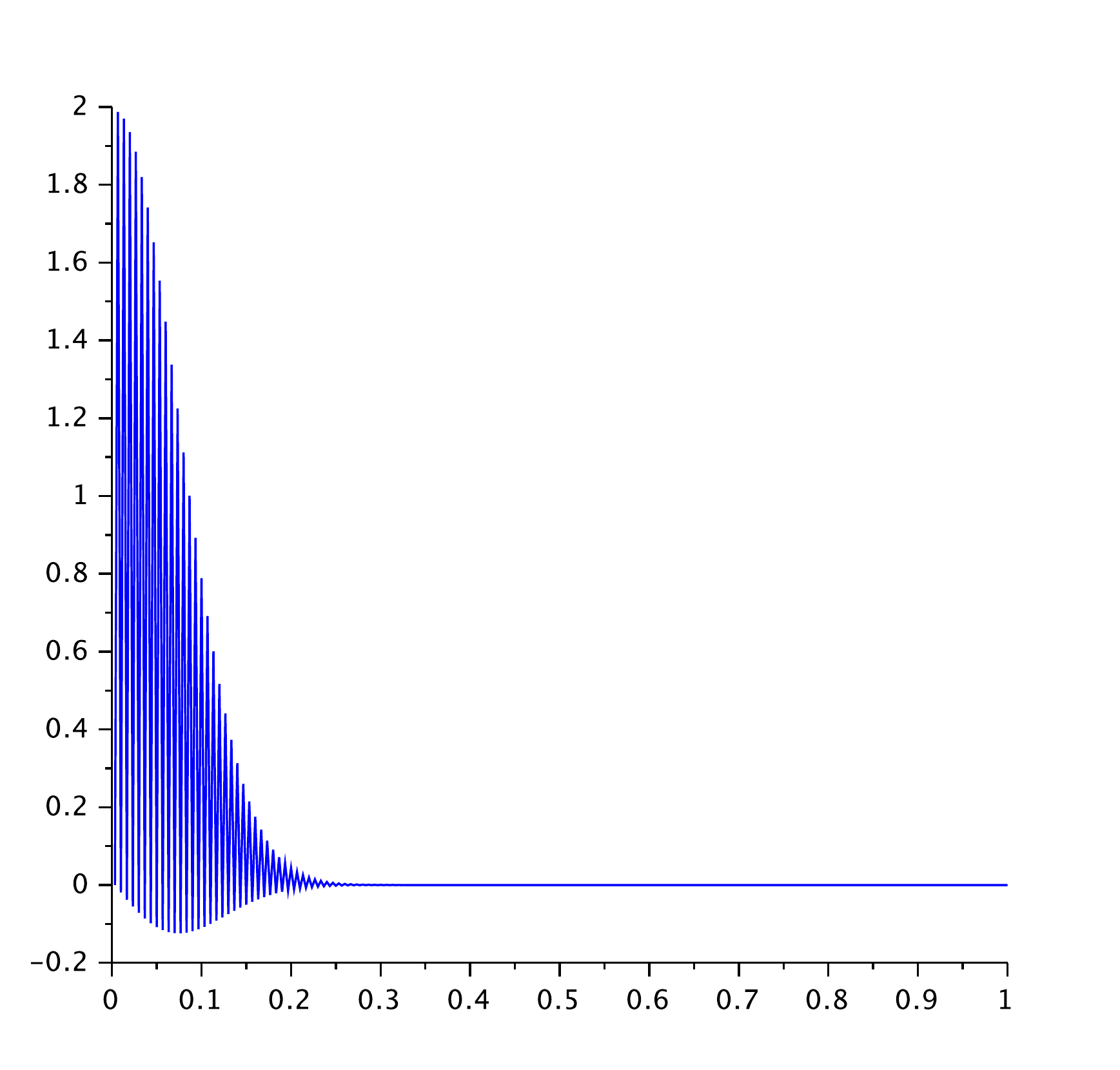}
\includegraphics[trim=20 20 20 20,clip,scale=0.35]{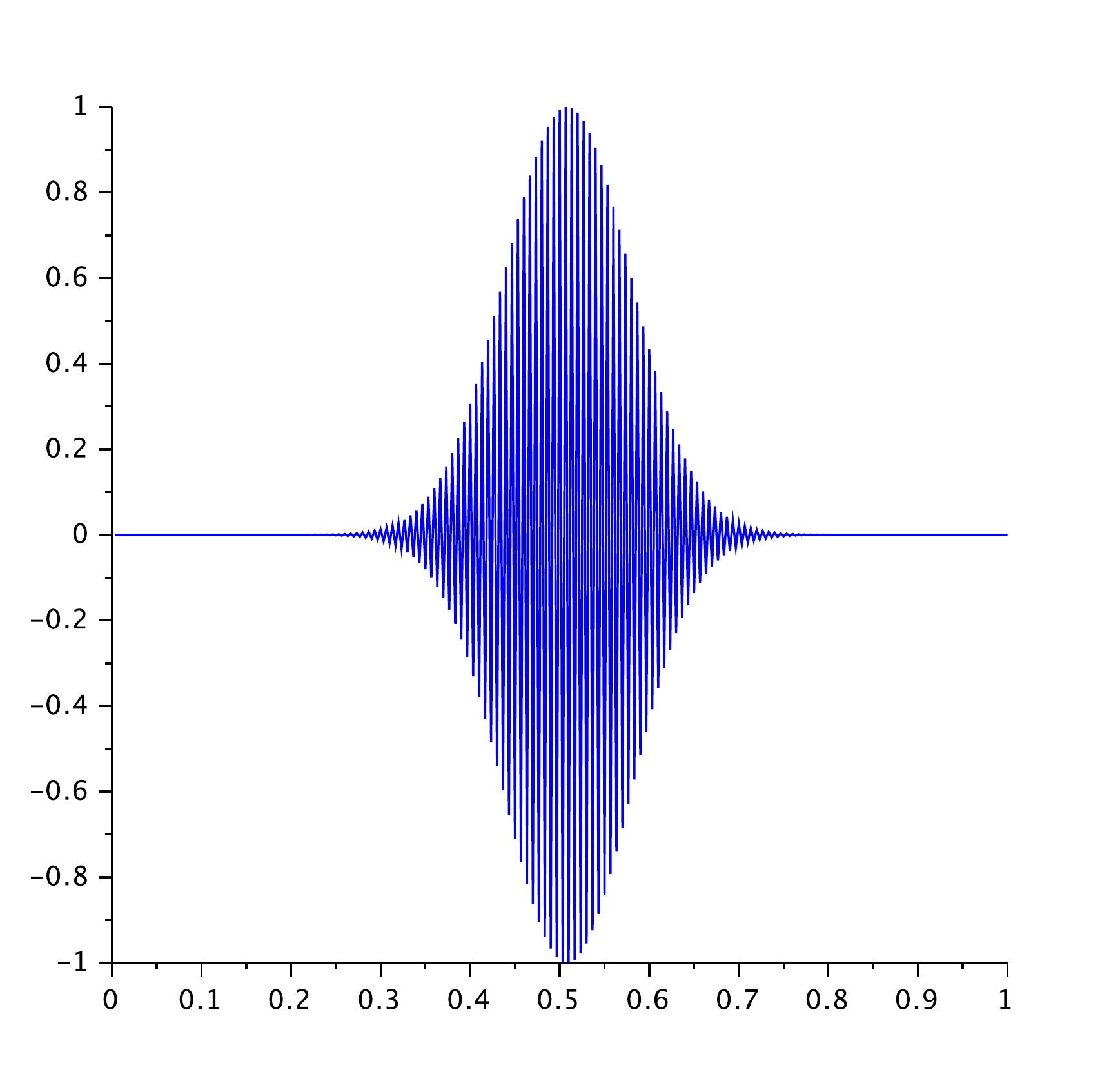}
\end{center}
\caption{Reflection of a bump by the leap-frog scheme with homogeneous Dirichlet condition at four successive times.}
\label{fig:SM}
\end{figure}

\paragraph{A scheme based on the backwards differentiation rule.}  We still start from the transport equation 
\eqref{transport}, approximate the space derivative $\partial_x u$ by the centered finite difference $(u_{j+1}-u_{j-1}) 
/(2\, \Delta x)$, and then apply the backwards differentiation formula of order $2$, see \cite[Chapter III.1]{hnw}. The 
resulting scheme reads:
\begin{equation*}
\dfrac{3}{2} \, u_j^{n+2} +\dfrac{\lambda \, a}{2} \, (u_{j+1}^{n+2}-u_{j-1}^{n+2}) -2 \, u_j^{n+1} +\dfrac{1}{2} \, u_j^n =0 \, .
\end{equation*}
This corresponds to $s=1$ and
\begin{equation*}
Q_2 \quad : \quad (u_j)_{j \in \Z} \longmapsto \left( \dfrac{3}{2} \, u_j +\dfrac{\lambda \, a}{2} \, (u_{j+1}-u_{j-1}) 
\right)_{j \in \Z} \, .
\end{equation*}
The operator $Q_2$ is an isomorphism on $\ell^2(\Z)$ since $Q_2$ is an isomorphism for any small $\lambda \, a$ 
(as a perturbation of $3/2 \, I$), $Q_2$ depends continuously on $\lambda \, a$, and there holds (uniformly with 
respect to $\lambda \, a$):
\begin{equation*}
\dfrac{3}{2} \, \Ng u \Nd_{-\infty,+\infty} \le \Ng Q_2 \, u \Nd_{-\infty,+\infty} \, .
\end{equation*}
The operator $Q_2$ is therefore an isomorphism on $\ell^2(\Z)$ for any $\lambda \, a>0$ (see, e.g., 
\cite[Lemma 4.3]{jfcsinum}). Let us now study the dispersion relation \eqref{dispersion}, which reads here
\begin{equation*}
\left( \dfrac{3}{2} +i \, \lambda \, a \, \sin \xi \right) \, z^2 -2 \, z +\dfrac{1}{2} =0 \, .
\end{equation*}
It is clear that the latter equation has two simple roots in $z$ for any $\xi \in \R$. Moreover, if $\sin \xi=0$, the roots 
are $1$ and $1/3$ which belong to $\Dbar$. In the case $\sin \xi \neq 0$, none of the roots belongs to $\cercle$ and 
examining the case $\lambda \, a \, \sin \xi=1$, we find that for $\sin \xi \neq 0$, both roots belong to $\D$ (which is 
consistent with the shape of the stability region for the backwards differentiation formula of order $2$, see 
\cite[Chapter V.1]{hw}). Assumption \ref{assumption1} is therefore satisfied. Assumption \ref{assumption2} is satisfied 
as long as $a$ is nonzero since there holds $p=r=1$ and $a_1(z)=a_{-1}(z)=\lambda \, a \, z^2/2$.

Theorem \ref{mainthm} therefore yields semigroup boundedness as long as one uses numerical boundary conditions 
for which the numerical scheme is well-defined (this is at least the case for $\lambda \, a$ small enough) and strong 
stability holds.

\subsubsection{Two space dimensions}

Here we wish to approximate the two-dimensional transport equation ($d=2$):
\begin{equation*}
\partial_t u +a_1 \, \partial_{x_1} u +a_2 \, \partial_{x_2} u =0 \, ,\quad u|_{t=0} =u_0 \, ,
\end{equation*}
in the space domain $\{ x_1>0 \, , \, x_2 \in \R \}$. When $a_1$ is negative, the latter problem does not 
necessitate any boundary condition at $x_1=0$. Following \cite{ag1}, we use one of the following two-dimensional 
versions of the leap-frog scheme, either
\begin{equation}
\label{lf2-1}
u_{j,k}^{n+2} +\lambda_1 \, a_1 \, (u_{j+1,k}^{n+1}-u_{j-1,k}^{n+1}) 
+\lambda_2 \, a_2 \, (u_{j,k+1}^{n+1}-u_{j,k-1}^{n+1}) -u_{j,k}^n =0 \, ,
\end{equation}
or
\begin{align}
u_{j,k}^{n+2} &+\lambda_1 \, a_1 \, \left( \dfrac{u_{j+1,k+1}^{n+1}+u_{j+1,k-1}^{n+1}}{2} 
-\dfrac{u_{j-1,k+1}^{n+1}+u_{j-1,k-1}^{n+1}}{2} \right) \notag \\
&+\lambda_2 \, a_2 \, \left( \dfrac{u_{j+1,k+1}^{n+1}+u_{j-1,k+1}^{n+1}}{2} 
-\dfrac{u_{j+1,k-1}^{n+1}+u_{j-1,k-1}^{n+1}}{2} \right) -u_{j,k}^n =0 \, .\label{lf2-2}
\end{align}
Assumption \ref{assumption0} is trivially satisfied because \eqref{lf2-1} and \eqref{lf2-2} are explicit schemes. 
The scheme \eqref{lf2-1} satisfies Assumption \ref{assumption1} if and only if $\lambda_1 \, |a_1| +\lambda_2 
\, |a_2|<1$, while the scheme \eqref{lf2-2} satisfies Assumption \ref{assumption1} if and only if $\max (\lambda_1 
\, |a_1|,\lambda_2 \, |a_2|)<1$. Let us now study when Assumption \ref{assumption2} is valid. For the scheme 
\eqref{lf2-1}, we have 
$r_1=p_1=1$, and
\begin{equation*}
a_1(z,\eta) =\lambda_1 \, a_1 \, z \, ,\quad a_{-1}(z,\eta) =-a_1(z,\eta) \, ,
\end{equation*}
so Assumption \ref{assumption2} is valid as long as $a_1 \neq 0$. For the scheme \eqref{lf2-2}, we have 
again $r_1=p_1=1$, and
\begin{equation*}
a_1(z,\eta) =z \, (\lambda_1 \, a_1 \, \cos \eta +i \, \lambda_2 \, a_2 \, \sin \eta) \, ,\quad 
a_{-1}(z,\eta) =z \, (-\lambda_1 \, a_1 \, \cos \eta +i \, \lambda_2 \, a_2 \, \sin \eta) \, ,
\end{equation*}
so Assumption \ref{assumption2} is valid as long as both $a_1$ and $a_2$ are nonzero. We refer to \cite{ag2} 
for the verification of strong stability depending on the choice of some numerical boundary conditions for 
\eqref{lf2-1} or \eqref{lf2-2}. Once again, if strong stability holds, then Theorem \ref{mainthm} yields 
semigroup boundedness and $\ell^2$-stability with respect to boundary data.

\section{The Leray-G{\aa}rding method for fully discrete Cauchy problems}
\label{section2}

This section is devoted to proving stability estimates for discretized Cauchy problems, which is the first step 
before considering the discretized initial boundary value problem \eqref{numibvp}. More precisely, we consider 
the simpler case of the whole space $j \in \Z^d$, and the recurrence relation:
\begin{equation}
\label{numcauchy}
\begin{cases}
{\dps \sum_{\sigma=0}^{s+1}} Q_\sigma \, u_j^{n+\sigma} =0 \, ,& 
j\in \Z^d \, ,\quad n\ge 0 \, ,\\
u_j^n = f_j^n \, ,& j\in \Z^d \, ,\quad n=0,\dots,s \, ,
\end{cases}
\end{equation}
where the operators $Q_\sigma$ are given by \eqref{defop}. We recall that in \eqref{defop}, the $a_{\ell,\sigma}$ 
are real numbers and are independent of the small parameter $\Delta t$ (they may depend on the CFL parameters 
$\lambda_1,\dots,\lambda_d$), while ${\bf S}$ denotes the shift operator on the space grid: $({\bf S}^\ell v)_j := 
v_{j+\ell}$ for $j,\ell \in \Z^d$. Stability of \eqref{numcauchy} is defined as follows.

\begin{definition}[Stability for the discrete Cauchy problem]
\label{def2}
The numerical scheme defined by \eqref{numcauchy} is ($\ell^2$-) stable if $Q_{s+1}$ is an isomorphism from 
$\ell^2 (\Z^d)$ onto itself, and if furthermore there exists a constant $C_0>0$ such that for all $\Delta t \in \, (0,1]$, 
for all initial conditions $(f_j^0)_{j \in \Z^d},\dots,(f_j^s)_{j \in \Z^d}$ in $\ell^2 (\Z^d)$, there holds
\begin{equation}
\label{estimcauchy}
\sup_{n \in \N} \, \Nd u^n \Nd_{-\infty,+\infty}^2 \le C_0 \, \sum_{\sigma=0}^s \Ng f^\sigma \Nd_{-\infty,+\infty}^2 \, .
\end{equation}
\end{definition}

\noindent Let us quickly recall that stability in the sense of Definition \ref{def2} is in fact independent of $\Delta t 
\in (0,1]$ (because \eqref{numcauchy} does not involve $\Delta t$ and \eqref{estimcauchy} can be simplified on 
either side by $\prod_i \Delta x_i$), and can be characterized in terms of the uniform power boundedness of the 
so-called amplification matrix
\begin{equation}
\label{defA2pas}
{\mathcal A}(\kappa) := \begin{pmatrix}
-\widehat{Q_s}(\kappa)/\widehat{Q_{s+1}}(\kappa) & \dots & \dots & -\widehat{Q_0}(\kappa)/\widehat{Q_{s+1}}(\kappa) \\
1 & 0 & \dots & 0 \\
0 & \ddots & \ddots & \vdots \\
0 & 0 & 1 & 0 \end{pmatrix} \in {\mathcal M}_{s+1}(\C)\, ,
\end{equation}
where the $\widehat{Q_\sigma}(\kappa)$'s are defined in \eqref{dispersion} and where it is understood that ${\mathcal A}$ 
is defined on the largest open set of $\C^d$ on which $\widehat{Q_{s+1}}$ does not vanish. Let us also recall that if 
$Q_{s+1}$ is an isomorphism from $\ell^2(\Z^d)$ onto itself, then $\widehat{Q_{s+1}}$ does not vanish on $(\cercle)^d$, 
and therefore does not vanish on an open neighborhood of $(\cercle)^d$. With the above definition \eqref{defA2pas} for 
${\mathcal A}$, the following well-known result holds:

\begin{proposition}[Characterization of stability for the fully discrete Cauchy problem]
\label{prop1}
Assume that $Q_{s+1}$ is an isomorphism from $\ell^2(\Z^d)$ onto itself. Then the scheme \eqref{numcauchy} is 
stable in the sense of Definition \ref{def2} if and only if there exists a constant $C_1>0$ such that the amplification 
matrix ${\mathcal A}$ in \eqref{defA2pas} satisfies
\begin{equation}
\label{unifpowerbound2}
\forall \, n \in \N \, ,\quad \forall \, \xi \in \R^d \, ,\quad 
\left| {\mathcal A}({\rm e}^{i\, \xi_1},\dots,{\rm e}^{i\, \xi_d})^n \right| \le C_1 \, .
\end{equation}
In particular, the spectral radius of ${\mathcal A}({\rm e}^{i\, \xi_1},\dots,{\rm e}^{i\, \xi_d})$ should not be 
larger than $1$ (the so-called von Neumann condition).
\end{proposition}

The eigenvalues of ${\mathcal A}({\rm e}^{i\, \xi_1},\dots,{\rm e}^{i\, \xi_d})$ are the roots to the dispersion 
relation \eqref{dispersion}. When these roots are simple for all $\xi \in \R^d$, the von Neumann condition is 
both necessary and {\it sufficient} for stability of  \eqref{numcauchy}, see, e. g., \cite[Proposition 3]{jfcnotes}. 
Assumption \ref{assumption1} is therefore a way to assume that \eqref{numcauchy} is stable for the discrete 
Cauchy problem. Our goal is to derive the semigroup estimate \eqref{estimcauchy} not by applying Fourier 
transform to \eqref{numcauchy} and using uniform power boundedness of ${\mathcal A}$, but rather by 
multiplying the first equation in \eqref{numcauchy} by a suitable {\it local} multiplier. The analysis relies 
first on the simpler case where one only considers the time evolution and no additional space variable.

\subsection{Stable recurrence relations}

In this Paragraph, we consider sequences $(v^n)_{n \in \N}$ with values in $\C$. The index $n$ should be thought 
of as the discrete time variable, and we therefore introduce the new notation ${\bf T}$ for the shift operator on the 
time grid: $({\bf T}^m v)^n :=v^{n+m}$ for all $m,n\in \N$. We start with the following elementary but crucial Lemma, 
which is the analogue of \cite[Lemme 1.1]{garding}.

\begin{lemma}[The energy-dissipation balance law]
\label{lem1}
Let $P \in \C[X]$ be  a polynomial of degree $s+1$ whose roots are simple and located in $\Dbar$. Then 
there exists a positive definite Hermitian form $q_e$ on $\C^{s+1}$, and a nonnegative Hermitian form $q_d$ 
on $\C^{s+1}$, that both depend in a ${\mathcal C}^\infty$ way on $P$, such that for any sequence $(v^n)_{n \in \N}$ 
with values in $\C$, there holds
\begin{equation*}
\forall \, n \in \N \, , \, 2 \, \text{\rm Re} \Big( \overline{{\bf T} \, (P'({\bf T}) \, v^n)} \, P({\bf T}) \, v^n \Big) 
=(s+1) \, |P({\bf T}) \, v^n|^2 +({\bf T}-I) \, (q_e(v^n,\dots,v^{n+s})) +q_d(v^n,\dots,v^{n+s}) \, .
\end{equation*}
In particular, for all sequence $(v^n)_{n \in \N}$ that satisfies the recurrence relation
\begin{equation*}
\forall \, n \in \N \, ,\quad P({\bf T}) \, v^n =0 \, ,
\end{equation*}
the sequence $(q_e(v^n,\dots,v^{n+s}))_{n\in \N}$ is nonincreasing.
\end{lemma}

The fact that there exists a Hermitian norm on $\C^{s+1}$ that is nonincreasing along solutions to the 
recurrence relation is not new. In fact, it is easily seen to be a consequence of the Kreiss matrix Theorem, 
see \cite{strikwerda-wade}. However, the important point here is that we can construct a multiplier that yields 
directly the "energy boundedness" (or decay). The fact that the coefficients of this multiplier are integer 
multiples of the coefficients of $P$ will be crucial in the analysis of Section \ref{section3}, see also Proposition 
\ref{prop2} below.

\begin{proof}
We borrow some ideas from \cite[Lemme 1.1]{garding} and introduce the interpolation polynomials:
\begin{equation*}
\forall \, k=1,\dots,s+1 \, ,\quad P_k(X) :=a \, \prod_{j \neq k} (X-x_j) \, ,
\end{equation*}
where $x_1,\dots,x_{s+1}$ denote the roots of $P$, and $a \neq 0$ its dominant coefficient. Since the roots of 
$P$ are pairwise distinct, the $P_k$'s form a basis of $\C_s[X]$ and they depend in a ${\mathcal C}^\infty$ way 
on the coefficients of $P$. We have
\begin{equation*}
P'=\sum_{k=1}^{s+1} P_k \, .
\end{equation*}
We then consider a sequence $(v^n)_{n \in \N}$ with values in $\C$ and compute
\begin{align*}
2 \, \text{\rm Re} \Big( \overline{{\bf T} \, (P'({\bf T}) \, v^n)} \, P({\bf T}) \, v^n \Big) -& (s+1) \, |P({\bf T}) \, v^n|^2 \\
=& \sum_{k=1}^{s+1} \overline{{\bf T} \, (P_k({\bf T})) \, v^n} \, ({\bf T} -x_k) \, P_k({\bf T}) \, v^n 
+{\bf T} \, (P_k({\bf T}) \, v^n) \, ({\bf T} -\overline{x_k}) \, \overline{P_k({\bf T}) \, v^n} \\
&-\sum_{k=1}^{s+1} ({\bf T} -\overline{x_k}) (\overline{P_k({\bf T}) \, v^n}) \, 
({\bf T} -x_k) (P_k({\bf T}) \, v^n) \\
=& \sum_{k=1}^{s+1} ({\bf T} -|x_k|^2) \, |P_k({\bf T}) \, v^n|^2 \, .
\end{align*}
The conclusion follows by defining:
\begin{align}
\forall \, (w^0,\dots,w^s) \in \C^{s+1} \, ,\quad q_e(w^0,\dots,w^s) &:=\sum_{k=1}^{s+1} |P_k({\bf T}) \, w^0|^2 \, ,
\label{defqe}\\
q_d(w^0,\dots,w^s) &:=\sum_{k=1}^{s+1} (1-|x_k|^2) \, |P_k({\bf T}) \, w^0|^2 \, .\label{defqd}
\end{align}
The form $q_e$ is positive definite because the $P_k$'s form a basis of $\C_s[X]$. The form $q_d$ is 
nonnegative because the roots of $P$ are located in $\Dbar$. Both forms depend in a ${\mathcal C}^\infty$ 
way on the coefficients of $P$ because the roots of $P$ are simple.
\end{proof}

Lemma \ref{lem1} shows that the polynomial $P'$ yields the good multiplier ${\bf T} \, P'({\bf T}) \, v^n$ for 
the recurrence relation $P({\bf T}) \, v^n =0$. Of course, $P'$ is not the only possible choice, though it will 
be our favorite one in what follows. As in \cite[Lemme 1.1]{garding}, any polynomial of the form\footnote{The 
sign condition here on the coefficients $\alpha_k$ is the analogue of the separation condition for the roots in 
\cite{leray,garding}.}
\begin{equation*}
Q :=\sum_{k=1}^{s+1} \alpha_k \, P_k \, ,\quad \alpha_1,\dots,\alpha_{s+1} >0 \, ,
\end{equation*}
provides with an energy balance of the form
\begin{equation*}
2 \, \text{\rm Re } \Big( \overline{{\bf T} \, (Q({\bf T}) \, v^n)} \, P({\bf T}) \, v^n \Big) 
=(\alpha_1+\cdots+\alpha_{s+1}) \, |P({\bf T}) \, v^n|^2 +({\bf T}-I) \, (q_e(v^n,\dots,v^{n+s})) 
+q_d(v^n,\dots,v^{n+s}) \, ,
\end{equation*}
with suitable Hermitian forms $q_e,q_d$ that have the same properties as stated in Lemma \ref{lem1}.

\subsection{The energy-dissipation balance for finite difference schemes}

In this Paragraph, we consider the numerical scheme \eqref{numcauchy}. We introduce the following notation: 
\begin{equation}
\label{defLM}
L:= \sum_{\sigma=0}^{s+1} {\bf T}^\sigma \, Q_\sigma \, ,\quad 
M:= \sum_{\sigma=0}^{s+1} \sigma \, {\bf T}^\sigma \, Q_\sigma \, .
\end{equation}
Thanks to Fourier analysis, Lemma \ref{lem1} easily gives the following result:

\begin{proposition}[The energy-dissipation balance law]
\label{prop2}
Let Assumptions \ref{assumption0} and \ref{assumption1} be satisfied. Then there exist a continuous coercive 
quadratic form $E_0$ and a continuous nonnegative quadratic form $D_0$ on $\ell^2(\Z^d;\R)^{s+1}$ such that 
for all sequences $(v^n)_{n \in \N}$ with values in $\ell^2(\Z^d;\R)$ and for all $n \in \N$, there holds
\begin{equation*}
2 \, \langle M\, v^n,L\, v^n \rangle_{-\infty,+\infty} =(s+1) \, \Ng L\, v^n \Nd_{-\infty,+\infty}^2 
+({\bf T}-I) \, E_0(v^n,\dots,v^{n+s}) +D_0(v^n,\dots,v^{n+s}) \, .
\end{equation*}
In particular, for all initial data $f^0,\dots,f^s \in \ell^2(\Z^d;\R)$, the solution to \eqref{numcauchy} satisfies
\begin{equation*}
\sup_{n \in \N} \, E_0(v^n,\dots,v^{n+s}) \le E_0(f^0,\dots,f^s) \, ,
\end{equation*}
and \eqref{numcauchy} is ($\ell^2$-)stable.
\end{proposition}

\begin{proof}
We use the same notation $v^n$ for the sequence $(v_j^n)_{j \in \Z^d}$ and the corresponding step function on 
$\R^d$ whose value on the cell $[j_1 \, \Delta x_1,(j_1+1) \, \Delta x_1) \times \cdots \times [j_d \, \Delta x_d,(j_d+1) 
\, \Delta x_d)$ equals $v_j^n$. Then Plancherel Theorem gives
\begin{multline*}
2 \, \langle M\, v^n,L\, v^n \rangle_{-\infty,+\infty}-(s+1) \, \Ng L\, v^n \Nd_{-\infty,+\infty}^2 \\
=\int_{\R^d} 2 \, \text{\rm Re} \Big( \overline{{\bf T} \, (P_\zeta'({\bf T}) \, \widehat{v^n}(\xi))} \, P_\zeta({\bf T}) \, 
\widehat{v^n} (\xi) \Big) -(s+1) \, |P_\zeta({\bf T}) \, \widehat{v^n}(\xi)|^2 \, \dfrac{{\rm d}\xi}{(2\, \pi)^d} \, ,
\end{multline*}
where $\widehat{v^n}$ denotes the Fourier transform of $v^n$, and where we have let
\begin{equation*}
P_\zeta (z):=\sum_{\sigma=0}^{s+1} \widehat{Q_\sigma} \big( {\rm e}^{i\, \zeta_1},\dots,{\rm e}^{i\, \zeta_d} \big) 
\, z^\sigma \, ,\quad \zeta_j := \xi_j \, \Delta x_j \, ,
\end{equation*}
and $P'_\zeta(z)$ denotes the derivative of $P_\zeta$ with respect to $z$.

From Assumption \ref{assumption1}, we know that for all $\zeta \in \R^d$, $P_\zeta$ has degree $s+1$ and has 
$s+1$ simple roots in $\Dbar$. We can apply Lemma \ref{lem1} and get
\begin{multline*}
2 \, \langle M\, v^n,L\, v^n \rangle_{-\infty,+\infty}-(s+1) \, \Ng L\, v^n \Nd_{-\infty,+\infty}^2 \\
=\int_{\R^d} ({\bf T}-I) \, q_{e,\zeta}(\widehat{v^n}(\xi),\dots,\widehat{v^{n+s}}(\xi)) 
+q_{d,\zeta}(\widehat{v^n}(\xi),\dots,\widehat{v^{n+s}}(\xi)) \, \dfrac{{\rm d}\xi}{(2\, \pi)^d} \, ,
\end{multline*}
where $q_{e,\zeta},q_{d,\zeta}$ depend in a ${\mathcal C}^\infty$ way on $\zeta \in \R^d$ and are $2\, \pi$-periodic 
in each $\zeta_j$. Furthermore, $q_{e,\zeta}$ is positive definite and $q_{d,\zeta}$ is nonnegative. The conclusion of 
Proposition \ref{prop2} follows by a standard compactness argument for showing coercivity of $E_0$.
\end{proof}

\subsection{Examples}

The first basic example corresponds to the case $s=0$ for which the multiplier provided by Proposition 
\ref{prop2} is $Q_1 \, v_j^{n+1}$. In that case, the energy $E_0$ reads $\Ng Q_1\, v \Nd_{-\infty,+\infty}^2$ 
(recall that $Q_1$ is an isomorphism) and the energy-dissipation balance law is nothing but the trivial identity
\begin{align*}
2 \, \langle Q_1\, v^{n+1},Q_1\, v^{n+1} +Q_0 \, v^n \rangle_{-\infty,+\infty} =& \, 
\Ng Q_1\, v^{n+1} +Q_0 \, v^n \Nd_{-\infty,+\infty}^2 \\
&+\Ng Q_1\, v^{n+1} \Nd_{-\infty,+\infty}^2 -\Ng Q_0\, v^n \Nd_{-\infty,+\infty}^2 \, .
\end{align*}
The second line of this algebraic identity can be rewritten as
\begin{equation*}
\Ng Q_1\, v^{n+1} \Nd_{-\infty,+\infty}^2 -\Ng Q_1\, v^n \Nd_{-\infty,+\infty}^2 
+\Ng Q_1\, v^n \Nd_{-\infty,+\infty}^2 -\Ng Q_0\, v^n \Nd_{-\infty,+\infty}^2 \, ,
\end{equation*}
and $\ell^2$-stability for the Cauchy problem amounts to assuming that the operator norm of $Q_1^{-1} \, Q_0$ 
is not larger than $1$. Hence the dissipation term $\Ng Q_1\, v^n \Nd_{-\infty,+\infty}^2 
-\Ng Q_0\, v^n \Nd_{-\infty,+\infty}^2$ is nonnegative.
\bigskip

Let us now consider the leap-frog scheme in one space dimension, for which we have $s=1$ and
\begin{equation*}
L={\bf T}^2 +\lambda \, a \, {\bf T} \, ({\bf S}-{\bf S}^{-1}) -I \, .
\end{equation*}
The corresponding dispersion relation \eqref{dispersion} reduces to
\begin{equation*}
z^2 +2\, i \, \lambda \, a \, \sin \xi \, z-1 =0 \, .
\end{equation*}
For $\lambda \, |a|<1$, the latter equation has two simple roots $x_1(\xi),x_2(\xi)$ of modulus $1$. Following 
the previous analysis, see \eqref{defqe}-\eqref{defqd}, the form $q_{e,\zeta}$ is given by
\begin{equation*}
q_{e,\zeta}(w^0,w^1) =|w^1 -x_1(\zeta) \, w^0|^2 +|w^1 -x_2(\zeta) \, w^0|^2 
=2\, |w^0|^2 +2\, |w^1|^2 +4\, \lambda \, a \, \text{\rm Re} \big( i\, \sin \zeta \, \overline{w^1} \, w^0 \big) \, ,
\end{equation*}
and $q_{d,\zeta}$ is zero. The associated forms in Proposition \ref{prop2} are $D_0 \equiv 0$ and (recall here $d=1$):
\begin{equation*}
E_0(v^0,v^1) =2 \, \sum_{j \in \Z} \Delta x \, (v^0_j)^2 +2 \, \sum_{j \in \Z} \Delta x \, (v^1_j)^2 
+2\, \lambda \, a \, \sum_{j \in \Z} \Delta x \, (v^0_{j+1}-v^0_{j-1}) \, v^1_j \, .
\end{equation*}
The latter energy functional $E_0$ is coercive under the condition $\lambda \, |a|<1$, which is the necessary 
and sufficient condition of stability for the leap-frog scheme, and $E_0$ is conserved for solutions to the 
leap-frog scheme. The conservation of $E_0$ is usually proved by starting from the recurrence relation
\begin{equation*}
\forall \, j \in \Z \, ,\quad \forall \, n \in \N \, ,\quad 
u_j^{n+2} +\lambda \, a \, (u_{j+1}^{n+1}-u_{j-1}^{n+1}) -u_j^n =0 \, ,
\end{equation*}
using the multiplier $u_j^{n+2}+u_j^n$, and summing with respect to $j$. This is equivalent, for solutions to the 
leap-frog scheme, to what we propose here, since our multiplier reads
\begin{equation*}
M \, u_j^n =2\, u_j^{n+2} +\lambda \, a \, (u_{j+1}^{n+1}-u_{j-1}^{n+1}) 
=u_j^{n+2}+u_j^n +\underbrace{L\, u_j^n}_{=0} \, .
\end{equation*}
However, it will appear more clearly in Section \ref{section3} why our choice for $M \, u_j^n$ has a major advantage 
when considering initial boundary value problems.

Let us observe here that the energy functional $E_0$ is associated with a local energy density
\begin{equation*}
E_{0,j}(v^0,v^1) := 2\, (v^0_j)^2 +2\, (v^1_j)^2 +2 \, \lambda \, a \, (v^0_{j+1}-v^0_{j-1}) \, v^1_j \, .
\end{equation*}
This is very specific to the leap-frog scheme. In general, the coefficients of the Hermitian forms $q_{e,\zeta},q_{d,\zeta}$ 
are not trigonometric polynomials of $\zeta$ and therefore $E_0,D_0$ do not necessarily admit local densities. 
This is one main difference with \cite{leray,garding}.

\section{Semigroup estimates for fully discrete initial boundary value problems}
\label{section3}

We now turn to the proof of Theorem \ref{mainthm} for which we shall use the results of Section \ref{section2} 
as a toolbox. By linearity of \eqref{numibvp}, it is sufficient to prove Theorem \ref{mainthm} separately in the 
case $(f_j^0)=\cdots=(f_j^s)=0$, and in the case $(F_j^n)=0$, $(g_j^n)=0$. The latter case is the most difficult 
and requires the introduction of an auxiliary set of \lq \lq dissipative\rq \rq $\,$ boundary conditions. Solutions 
to \eqref{numibvp} are always assumed to be real valued, which means that the data are real valued. For 
complex valued initial data and/or forcing terms, one just uses the linearity of \eqref{numibvp}.

\subsection{The case with zero initial data}

We first assume $(f_j^0)=\cdots=(f_j^s)=0$. By strong stability, we already know that \eqref{stabilitenumibvp} 
holds with a constant $C$ that is independent of $\gamma>0$ and $\Delta t \in (0,1]$. Therefore, proving 
Theorem \ref{mainthm} amounts to showing the existence of a constant $C$, that is independent of $\gamma>0$ 
and $\Delta t \in (0,1]$ such that the solution to \eqref{numibvp} with $(f_j^0)=\cdots=(f_j^s)=0$ satisfies
\begin{multline}
\label{estimnumibvp'}
\sup_{n \ge 0} \, {\rm e}^{-2\, \gamma \, n\, \Delta t} \, \Ng u^n \Nd_{1-r_1,+\infty}^2 
\le C \, \left\{ \dfrac{\gamma \, \Delta t+1}{\gamma} \, 
\sum_{n\ge s+1} \Delta t \, {\rm e}^{-2\, \gamma \, n\, \Delta t} \, \Ng F^n \Nd_{1,+\infty}^2 \right. \\
\left. +\sum_{n\ge s+1} \Delta t \, {\rm e}^{-2\, \gamma \, n\, \Delta t} \, \sum_{j_1=1-r_1}^0 
\| g_{j_1,\cdot}^n \|_{\ell^2(\Z^{d-1})}^2 \right\} \, .
\end{multline}
We thus consider a parameter $\gamma>0$ and a time step $\Delta t \in (0,1]$, and focus on the numerical scheme 
\eqref{numibvp} with zero initial data (that is, $(f_j^0)=\cdots=(f_j^s)=0$). For all $n \in \N$, we extend the sequence 
$(u_j^n)$ by zero for $j_1 \le -r_1$:
\begin{equation*}
v_j^n :=\begin{cases}
u_j^n &\text{\rm if } j_1 \ge 1-r_1 \, ,\\
0 &\text{\rm otherwise.}
\end{cases}
\end{equation*}
We use Proposition \ref{prop2} and compute:
\begin{equation*}
({\bf T}-I) \, E_0(v^n,\dots,v^{n+s}) +D_0(v^n,\dots,v^{n+s}) =
2 \, \langle M\, v^n,L\, v^n \rangle_{-\infty,+\infty} -(s+1) \, \Ng L\, v^n \Nd_{-\infty,+\infty}^2  \, .
\end{equation*}
Due to the form of the operator $L$, see \eqref{defLM}, and the fact that $v_j^n$ vanishes for $j_1 \le -r_1$, there 
holds:
\begin{equation*}
L\, v_j^n =\begin{cases}
\Delta t \, F_j^{n+s+1} &\text{\rm if } j_1 \ge 1 \, ,\\
0 &\text{\rm if } j_1 \le -r_1-p_1 \, ,
\end{cases}
\end{equation*}
and we thus get
\begin{align*}
({\bf T}-I) \, E_0(v^n,\dots,v^{n+s}) \, +& \, D_0(v^n,\dots,v^{n+s}) \\
=& \left( \prod_{k=1}^d \Delta x_k \right) \, \sum_{j_1 \ge 1} \sum_{j' \in \Z^{d-1}} 
2 \, \Delta t \, (M \, v_j^n) \, F_j^{n+s+1} -(s+1) \, \Delta t^2 \, (F_j^{n+s+1})^2 \\
&+ \left( \prod_{k=1}^d \Delta x_k \right) \, \sum_{j_1=1-r_1-p_1}^0 \sum_{j' \in \Z^{d-1}} 
2 \, (M \, v_j^n) \, L\, v_j^n -(s+1) \, (L\, v_j^n)^2 \, .
\end{align*}
We multiply the latter equality by $\exp(-2\, \gamma \, (n+s+1) \, \Delta t)$, sum with respect to $n$ 
from $0$ to some $N$ and use the fact that $D_0$ is nonnegative. Recalling that the initial data in 
\eqref{numibvp} vanish, we get
\begin{multline}
\label{energy1}
{\rm e}^{-2\, \gamma \, (N+s+1) \, \Delta t} \, E_0 \big( v^{N+1},\dots,v^{N+s+1} \big) 
+\underbrace{\big( 1-{\rm e}^{-2\, \gamma \, \Delta t} \big) \, \sum_{n=1}^N {\rm e}^{-2\, \gamma \, (n+s) \, \Delta t} \, 
E_0 (v^n,\dots,v^{n+s})}_{\ge 0} \\
\le S_{1,N} +S_{2,N} \, ,
\end{multline}
with
\begin{equation}
\label{defS1Ngamma}
S_{1,N} :=\sum_{n=0}^N {\rm e}^{-2\, \gamma \, (n+s+1) \, \Delta t} \, \Big( 2 \, \Delta t \, 
\langle M\, v^n,F^{n+s+1} \rangle_{1,+\infty} -(s+1) \, \Delta t^2 \, \Ng F^{n+s+1} \Nd_{1,+\infty}^2 \Big) \, ,
\end{equation}
and
\begin{equation}
\label{defS2Ngamma}
S_{2,N} :=\left( \prod_{k=1}^d \Delta x_k \right) \, \sum_{n=0}^N {\rm e}^{-2\, \gamma \, (n+s+1) \, \Delta t} 
\, \sum_{j_1=1-r_1-p_1}^0 \sum_{j' \in \Z^{d-1}} 2 \, (M \, v_j^n) \, L\, v_j^n -(s+1) \, (L\, v_j^n)^2 \, .
\end{equation}

Let us now estimate the two source terms $S_{1,N},S_{2,N}$ in \eqref{energy1}. We begin with the term $S_{2,N}$ 
defined in \eqref{defS2Ngamma}. Let us recall that the ratio $\Delta t/\Delta x_1$ is fixed. Furthermore, the form of 
the operators $L$ and $M$ in \eqref{defLM} gives the estimate (recall that $v_j^n$ vanishes for $j_1 \le -r_1$):
\begin{equation*}
S_{2,N} \le C \, \Delta t \, \left( \prod_{k=2}^d \Delta x_k \right) \, \sum_{n=0}^N 
{\rm e}^{-2\, \gamma \, (n+s+1) \, \Delta t} \, \sum_{j_1=1-r_1}^{p_1} \sum_{j' \in \Z^{d-1}} 
(u_j^n)^2 +\cdots +(u_j^{n+s+1})^2 \, ,
\end{equation*}
for a constant $C$ that does not depend on $N$, $\gamma$ nor on $\Delta t$. We thus have, uniformly with respect 
to $N \in \N$, $\gamma>0$ and $\Delta t \in (0,1]$:
\begin{align*}
S_{2,N} &\le C \, \sum_{n=s+1}^{N+s+1} \Delta t \, {\rm e}^{-2\, \gamma \, n \, \Delta t} \, 
\sum_{j_1=1-r_1}^{p_1} \| u_{j_1,\cdot}^n \|_{\ell^2 (\Z^{d-1})}^2 
\le C \, \sum_{n \ge s+1} \Delta t \, {\rm e}^{-2\, \gamma \, n \, \Delta t} \, 
\sum_{j_1=1-r_1}^{p_1} \| u_{j_1,\cdot}^n \|_{\ell^2 (\Z^{d-1})}^2 \\
&\le C \, \left\{ \dfrac{\gamma \, \Delta t+1}{\gamma} \, 
\sum_{n\ge s+1} \Delta t \, {\rm e}^{-2\, \gamma \, n\, \Delta t} \, \Ng F^n \Nd_{1,+\infty}^2 
+\sum_{n\ge s+1} \Delta t \, {\rm e}^{-2\, \gamma \, n\, \Delta t} \, \sum_{j_1=1-r_1}^0 
\| g_{j_1,\cdot}^n \|_{\ell^2(\Z^{d-1})}^2 \right\} \, ,
\end{align*}
where we have used the trace estimate \eqref{stabilitenumibvp} that follows from the strong stability assumption.

Let us now focus on the term $S_{1,N}$ in \eqref{energy1}, see the defining equation \eqref{defS1Ngamma}. 
We use the Cauchy-Schwarz inequality and derive (using now the interior estimate in \eqref{stabilitenumibvp} 
that follows from the strong stability assumption):
\begin{align*}
S_{1,N} &\le 2 \, \sum_{n=0}^N \Delta t \, {\rm e}^{-2\, \gamma \, (n+s+1) \, \Delta t} \, 
\Ng M\, v^n \Nd_{1,+\infty} \, \Ng F^{n+s+1} \Nd_{1,+\infty} \\
&\le C \, \sum_{n=0}^N \Delta t \, {\rm e}^{-2\, \gamma \, (n+s+1) \, \Delta t} \, \Big( 
\Ng v^{n+1} \Nd_{1-r_1,+\infty} +\cdots +\Ng v^{n+s+1} \Nd_{1-r_1,+\infty} \Big) \, \Ng F^{n+s+1} \Nd_{1,+\infty} \\
&\le C\, \dfrac{\gamma}{\gamma \, \Delta t+1} \, 
\sum_{n=s+1}^{N+s+1} \Delta t \, {\rm e}^{-2\, \gamma \, n \, \Delta t} \, \Ng u^n \Nd_{1-r_1,+\infty}^2 
+C \, \dfrac{\gamma \, \Delta t+1}{\gamma} \, \sum_{n=s+1}^{N+s+1} \Delta t \, 
{\rm e}^{-2\, \gamma \, n \, \Delta t} \, \Ng F^n \Nd_{1,+\infty}^2 \\
&\le C \, \left\{ \dfrac{\gamma \, \Delta t+1}{\gamma} \, 
\sum_{n\ge s+1} \Delta t \, {\rm e}^{-2\, \gamma \, n\, \Delta t} \, \Ng F^n \Nd_{1,+\infty}^2 
+\sum_{n\ge s+1} \Delta t \, {\rm e}^{-2\, \gamma \, n\, \Delta t} \, \sum_{j_1=1-r_1}^0 
\| g_{j_1,\cdot}^n \|_{\ell^2(\Z^{d-1})}^2 \right\} \, .
\end{align*}
Ignoring the nonnegative term on the left hand-side of \eqref{energy1} and using the coercivity of $E_0$, we 
have proved that there exists a constant $C>0$ that is uniform with respect to $N,\gamma,\Delta t$ such that:
\begin{multline*}
{\rm e}^{-2\, \gamma \, (N+s+1) \, \Delta t} \, \Ng v^{N+s+1} \Nd_{-\infty,+\infty}^2 \le C \, 
\left\{ \dfrac{\gamma \, \Delta t+1}{\gamma} \, 
\sum_{n\ge s+1} \Delta t \, {\rm e}^{-2\, \gamma \, n\, \Delta t} \, \Ng F^n \Nd_{1,+\infty}^2 \right. \\
\left. +\sum_{n\ge s+1} \Delta t \, {\rm e}^{-2\, \gamma \, n\, \Delta t} \, \sum_{j_1=1-r_1}^0 
\| g_{j_1,\cdot}^n \|_{\ell^2(\Z^{d-1})}^2 \right\} \, ,
\end{multline*}
which yields \eqref{estimnumibvp'} and therefore the validity of Theorem \ref{mainthm} in the case of zero 
initial data.

\subsection{Construction of dissipative boundary conditions}

In this paragraph, we consider an auxiliary problem for which we shall be able to prove simultaneously 
an optimal semigroup estimate and a trace estimate for the solution. More precisely, we shall prove the 
following result.

\begin{theorem}
\label{absorbing}
Let Assumptions \ref{assumption0}, \ref{assumption1} and \ref{assumption2} be satisfied. Then for all $P_1 
\in \N$, there exists a constant $C_{P_1}>0$ such that, for all initial data $(f^0_j),\dots,(f^s_j) \in \ell^2 (\Z^d)$ 
and for all source term $(g_j^n)_{j_1 \le 0,n \ge s+1}$ that satisfies
\begin{equation*}
\forall \, \Gamma >0 \, ,\quad \sum_{n\ge s+1} {\rm e}^{-2\, \Gamma \, n} \, \sum_{j_1 \le 0} 
\| g_{j_1,\cdot}^n \|_{\ell^2(\Z^{d-1})}^2 < +\infty \, ,
\end{equation*}
there exists a unique sequence $(u_j^n)_{j \in \Z^d,n\in \N}$ solution to
\begin{equation}
\label{numabsorbing}
\begin{cases}
L \, u_j^n =0 \, ,& j \in \Z^d \, ,\quad j_1 \ge 1\, ,\quad n\ge 0 \, ,\\
M \, u_j^n =g_j^{n+s+1} \, ,& j \in \Z^d \, ,\quad j_1 \le 0\, ,\quad n\ge 0 \, ,\\
u_j^n = f_j^n \, ,& j \in \Z^d \, ,\quad n=0,\dots,s \, .
\end{cases}
\end{equation}
Moreover for all $\gamma>0$ and $\Delta t \in (0,1]$, this solution satisfies
\begin{multline}
\label{estimabsorb}
\sup_{n \ge 0} \, {\rm e}^{-2\, \gamma \, n\, \Delta t} \, \Ng u^n \Nd_{-\infty,+\infty}^2 
+\dfrac{\gamma}{\gamma \, \Delta t+1} \, 
\sum_{n\ge 0} \Delta t \, {\rm e}^{-2\, \gamma \, n\, \Delta t} \, \Ng u^n \Nd_{-\infty,+\infty}^2 
+\sum_{n\ge 0} \Delta t \, {\rm e}^{-2\, \gamma \, n\, \Delta t} \, \sum_{j_1=1-r_1}^{P_1} 
\| u_{j_1,\cdot}^n \|_{\ell^2(\Z^{d-1})}^2 \\
\le C_{P_1} \, \left\{ \sum_{\sigma=0}^s \Ng f^\sigma \Nd_{-\infty,+\infty}^2 
+\sum_{n\ge s+1} \Delta t \, {\rm e}^{-2\, \gamma \, n\, \Delta t} \, \sum_{j_1 \le 0} 
\| g_{j_1,\cdot}^n \|_{\ell^2(\Z^{d-1})}^2 \right\} \, .
\end{multline}
\end{theorem}

Theorem \ref{absorbing} justifies why we advocate the choice $M\, u_j^n =2\, u_j^{n+2} +\lambda \, a \, 
(u_{j+1}^{n+1} -u_{j-1}^{n+1})$ rather than the more standard $u_j^{n+2}+u_j^n$ as a multiplier for the 
leap-frog scheme. Despite repeated efforts, we have not been able to prove the estimate \eqref{estimabsorb} 
when using the numerical boundary condition $u_j^{n+2}+u_j^n$ on $j_1 \le 0$, in conjunction with the 
leap-frog scheme on $j_1 \ge 1$.

\begin{proof}
Let us first quickly observe that the solution to \eqref{numabsorbing} is well-defined since, as long as we have 
determined the solution up to a time index $n+s$, $n \ge 0$, then $u^{n+s+1}$ is sought as a solution to an 
equation of the form
\begin{equation*}
Q_{s+1} \, u^{n+s+1} =F \, ,
\end{equation*}
where $F$ belongs to $\ell^2(\Z^d)$ (this is due to the form of $L$ and $M$, see \eqref{defLM}). Hence $u^n$ 
is uniquely defined and belongs to $\ell^2(\Z^d)$ for all $n \in \N$.

The proof of Theorem \ref{absorbing} starts again with the application of Proposition \ref{prop2}. Using 
the nonnegativity of the dissipation form $D_0$, we get\footnote{Since $L\, u_j^n=0$ for $j_1 \ge 1$, one 
could also write $\Ng L\, u^n \Nd_{-\infty,0}^2$ rather than $\Ng L\, u^n \Nd_{-\infty,+\infty}^2$ on the left 
hand-side of the inequality.}
\begin{equation*}
({\bf T}-I) \, E_0(u^n,\dots,u^{n+s}) +(s+1) \, \Ng L\, u^n \Nd_{-\infty,+\infty}^2 
\le 2 \, \langle M\, u^n,L\, u^n \rangle_{-\infty,+\infty} =2 \, \langle g^{n+s+1},L\, u^n \rangle_{-\infty,0} \, .
\end{equation*}
By the Young inequality, we get
\begin{equation*}
({\bf T}-I) \, E_0(u^n,\dots,u^{n+s}) +\dfrac{s+1}{2} \, \Ng L\, u^n \Nd_{-\infty,+\infty}^2 
\le \dfrac{2}{s+1} \, \Ng g^{n+s+1} \Nd_{-\infty,0}^2 \, .
\end{equation*}
We multiply the latter inequality by $\exp (-2\, \gamma \, (n+s+1) \, \Delta t)$, sum from $n=0$ to some arbitrary 
$N$ and already derive the estimate (here we use again the fact that $\Delta t/\Delta x_1$ is a fixed positive 
constant):
\begin{multline*}
\sup_{n \ge 1} \, {\rm e}^{-2\, \gamma \, (n+s)\, \Delta t} \, E_0(u^n,\dots,u^{n+s}) 
+\big( 1-{\rm e}^{-2\, \gamma \, \Delta t} \big) \, 
\sum_{n \ge 0} {\rm e}^{-2\, \gamma \, (n+s)\, \Delta t} \, E_0(u^n,\dots,u^{n+s}) \\
+\sum_{n\ge 0} \Delta t \, {\rm e}^{-2\, \gamma \, (n+s+1)\, \Delta t} \, \sum_{j_1 \in \Z} 
\| L\, u_{j_1,\cdot}^n \|_{\ell^2(\Z^{d-1})}^2 \\
\le C \, \left\{ {\rm e}^{-2\, \gamma \, s\, \Delta t} \, E_0(f^0,\dots,f^s) 
+\sum_{n\ge s+1} \Delta t \, {\rm e}^{-2\, \gamma \, n\, \Delta t} \, \sum_{j_1 \le 0} 
\| g_{j_1,\cdot}^n \|_{\ell^2(\Z^{d-1})}^2 \right\} \, .
\end{multline*}
Using the coercivity of $E_0$ and the inequality 
\begin{equation*}
1-{\rm e}^{-2\, \gamma \, \Delta t} \ge \dfrac{\gamma \, \Delta t}{\gamma \, \Delta t +1} \, ,
\end{equation*}
we have therefore derived the estimate
\begin{multline}
\label{estim1}
\sup_{n \ge 0} \, {\rm e}^{-2\, \gamma \, n\, \Delta t} \, \Ng u^n \Nd_{-\infty,+\infty}^2 
+\dfrac{\gamma}{\gamma \, \Delta t +1} \, \sum_{n\ge 0} \Delta t \, {\rm e}^{-2\, \gamma \, n\, \Delta t} \, 
\Ng u^n \Nd_{-\infty,+\infty}^2 \\
+\sum_{n\ge 0} \Delta t \, {\rm e}^{-2\, \gamma \, (n+s+1)\, \Delta t} \, \sum_{j_1 \in \Z} 
\| L\, u_{j_1,\cdot}^n \|_{\ell^2(\Z^{d-1})}^2 \\
\le C \, \left\{ \sum_{\sigma=0}^s \Ng f^\sigma \Nd_{-\infty,+\infty}^2 
+\sum_{n\ge s+1} \Delta t \, {\rm e}^{-2\, \gamma \, n\, \Delta t} \, \sum_{j_1 \le 0} 
\| g_{j_1,\cdot}^n \|_{\ell^2(\Z^{d-1})}^2 \right\} \, ,
\end{multline}
where the constant $C$ is independent of $\gamma$, $\Delta t$ and on the solution $(u_j^n)$. In order to 
prove \eqref{estimabsorb}, the main remaining task is to derive the trace estimate for $(u_j^n)$. This is done 
by first dealing with the case where $\gamma \, \Delta t$ is large.
\bigskip

$\bullet$ From the definition of the operator $L$, see \eqref{defLM}, there exists a constant $C>0$ and an 
integer $J$ such that
\begin{equation*}
(L\, u_j^n)^2 \ge \dfrac{1}{2} \, (Q_{s+1} \, u_j^{n+s+1})^2 -C \, \sum_{\sigma=0}^s \sum_{|\ell| \le J} 
(u_{j+\ell}^{n+\sigma})^2 \, .
\end{equation*}
Since $Q_{s+1}$ is an isomorphism, there exists a constant $c>0$ such that
\begin{equation*}
\sum_{j \in \Z^d} (L\, u_j^n)^2 \ge c \, \sum_{j \in \Z^d} (u_j^{n+s+1})^2 -\dfrac{1}{c} \, \sum_{\sigma=0}^s 
\sum_{j \in \Z^d} (u_j^{n+\sigma})^2 \, .
\end{equation*}
Multiplying by $\exp (-2 \, \gamma \, (n+s+1) \, \Delta t)$ and summing with respect to $n \in \N$, we get
\begin{multline*}
\sum_{n\ge s+1} \Delta t \, {\rm e}^{-2\, \gamma \, n\, \Delta t} \, \sum_{j_1 \in \Z} 
\| u_{j_1,\cdot}^n \|_{\ell^2(\Z^{d-1})}^2 \le C \, \left\{ 
\sum_{n\ge 0} \Delta t \, {\rm e}^{-2\, \gamma \, (n+s+1)\, \Delta t} \, \sum_{j_1 \in \Z} 
\| L\, u_{j_1,\cdot}^n \|_{\ell^2(\Z^{d-1})}^2 \right. \\
\left. +{\rm e}^{-2\, \gamma \, \Delta t} \, \sum_{n\ge 0} \Delta t \, {\rm e}^{-2\, \gamma \, n\, \Delta t} \, 
\sum_{j_1 \in \Z} \| u_{j_1,\cdot}^n \|_{\ell^2(\Z^{d-1})}^2 \right\} \, .
\end{multline*}
Choosing $\gamma \, \Delta t$ large enough, that is $\gamma \, \Delta t \ge \ln R_0$ for some numerical constant 
$R_0>1$ that depends only on the (fixed) coefficients of the operator $L$, we have derived the estimate
\begin{multline*}
\sum_{n\ge s+1} \Delta t \, {\rm e}^{-2\, \gamma \, n\, \Delta t} \, \sum_{j_1 \in \Z} 
\| u_{j_1,\cdot}^n \|_{\ell^2(\Z^{d-1})}^2 \le C \, \left\{ 
\sum_{n\ge 0} \Delta t \, {\rm e}^{-2\, \gamma \, (n+s+1)\, \Delta t} \, \sum_{j_1 \in \Z} 
\| L\, u_{j_1,\cdot}^n \|_{\ell^2(\Z^{d-1})}^2 \right. \\
\left. +{\rm e}^{-2\, \gamma \, \Delta t} \, \sum_{\sigma=0}^s \Ng f^\sigma \Nd_{-\infty,+\infty}^2 \right\} \, .
\end{multline*}
It remains to use \eqref{estim1} and we get an even better estimate than \eqref{estimabsorb} which we were 
originally aiming at:
\begin{equation*}
\sum_{n\ge s+1} \Delta t \, {\rm e}^{-2\, \gamma \, n\, \Delta t} \, \sum_{j_1 \in \Z} 
\| u_{j_1,\cdot}^n \|_{\ell^2(\Z^{d-1})}^2 \le C \, \left\{ \sum_{\sigma=0}^s \Ng f^\sigma \Nd_{-\infty,+\infty}^2 
+\sum_{n\ge s+1} \Delta t \, {\rm e}^{-2\, \gamma \, n\, \Delta t} \, \sum_{j_1 \le 0} 
\| g_{j_1,\cdot}^n \|_{\ell^2(\Z^{d-1})}^2 \right\} \, .
\end{equation*}
This gives a control of infinitely many traces and not only finitely many (this restriction to finitely many traces 
will appear in the regime where $\gamma \, \Delta t$ can be small).
\bigskip

$\bullet$ From now on, we have fixed a constant $R_0>1$ such that \eqref{estimabsorb} holds for $\gamma \, 
\Delta t \ge \ln R_0$ and we thus assume $\gamma \, \Delta t \in (0,\ln R_0]$. We also know that the estimate 
\eqref{estim1} holds, independently of the value of $\gamma \, \Delta t$, and we now wish to estimate the traces 
of the solution $(u_j^n)$ for finitely many values of $j_1$.

We first observe from \eqref{estim1} that for all $\gamma>0$ and $\Delta t \in (0,1]$, there exists a constant 
$C_{\gamma,\Delta t}$ such that
\begin{equation*}
\forall \, n \in \N \, ,\quad {\rm e}^{-2\, \gamma \, n\, \Delta t} \, \sum_{j \in \Z^d} (u_j^n)^2 \le C_{\gamma,\Delta t} \, .
\end{equation*}
In particular, for any $j_1 \in \Z$, the Laplace-Fourier transforms $\widehat{u_{j_1}}$ of the step functions
\begin{equation*}
u_{j_1} \quad : \quad (t,y) \in \R^+ \times \R^{d-1} \mapsto u_j^n \quad \text{\rm if } (t,y) \in [n\, \Delta t,(n+1) \, \Delta t) 
\times \prod_{k=2}^d  [j_k \, \Delta x_k,(j_k+1) \, \Delta x_k) \, ,
\end{equation*}
is well-defined on $\{ \tau \in \C \, , \, \text{\rm Re } \tau >0 \} \times \R^{d-1}$. The dual variables are denoted 
$\tau =\gamma +i\, \theta$, $\gamma>0$, and $\eta =(\eta_2,\dots,\eta_d) \in \R^{d-1}$. It will also be convenient 
to introduce the notation $\eta_\Delta := (\eta_2 \, \Delta x_2,\dots,\eta_d \, \Delta x_d)$. Given $\Gamma>0$, the 
sequence $(\widehat{u_{j_1}}(\Gamma+i\, \theta,\eta))_{j_1 \in \Z}$ belongs to $\ell^2(\Z)$ for almost every 
$(\theta,\eta) \in \R \times \R^{d-1}$.

We first show the following estimate.

\begin{lemma}
\label{lem3'}
With $R_0>1$ fixed as above, there exists a constant $C>0$ such that for all $\gamma>0$ and $\Delta t \in (0,1]$ 
satisfying $\gamma \, \Delta t \in (0,\ln R_0]$, there holds
\begin{multline}
\label{estimlem3'}
\sum_{j_1 \in \Z} \int_{\R \times \R^{d-1}} \left| \sum_{\ell_1=-r_1}^{p_1} 
a_{\ell_1} \big( {\rm e}^{(\gamma +i\, \theta) \, \Delta t},\eta_\Delta \big) \, 
\widehat{u_{j_1+\ell_1}}(\gamma +i\, \theta,\eta) \right|^2 \, {\rm d}\theta \, {\rm d}\eta \\
+\sum_{j_1 \le 0} \int_{\R \times \R^{d-1}} \left| \sum_{\ell_1=-r_1}^{p_1} {\rm e}^{(\gamma +i\, \theta) \, \Delta t} 
\, \partial_z a_{\ell_1} \big( {\rm e}^{(\gamma +i\, \theta) \, \Delta t},\eta_\Delta \big) \, 
\widehat{u_{j_1+\ell_1}}(\gamma +i\, \theta,\eta) \right|^2 \, {\rm d}\theta \, {\rm d}\eta \\
\le C \, \left\{ \sum_{\sigma=0}^s \Ng f^\sigma \Nd_{-\infty,+\infty}^2 
+\sum_{n\ge s+1} \Delta t \, {\rm e}^{-2\, \gamma \, n\, \Delta t} \, \sum_{j_1 \le 0} 
\| g_{j_1,\cdot}^n \|_{\ell^2(\Z^{d-1})}^2 \right\} \, .
\end{multline}
\end{lemma}

\begin{proof}[Proof of Lemma \ref{lem3'}]
Given $\tau =\gamma +i \, \theta$ and $\eta$, we compute (here $j_1 \in \Z$ is fixed):
\begin{align}
\sum_{\ell_1=-r_1}^{p_1} a_{\ell_1} \big( {\rm e}^{\tau \, \Delta t},\eta_\Delta \big) 
\, \widehat{u_{j_1+\ell_1}}(\tau,\eta)  &= \widehat{L\, u_{j_1,\cdot}} (\tau,\eta) 
+\dfrac{1-{\rm e}^{-\tau \, \Delta t}}{\tau} \, 
\sum_{\sigma=1}^{s+1} \sum_{\sigma'=0}^{\sigma-1} {\rm e}^{(\sigma-\sigma') \, \tau \, \Delta t} 
\, {\mathcal F}_{j_1}^{\sigma,\sigma'}(\eta) \, ,\label{lem3'1} \\
\sum_{\ell_1=-r_1}^{p_1} {\rm e}^{\tau \, \Delta t} \, \partial_z a_{\ell_1} \big( {\rm e}^{\tau \, \Delta t},\eta_\Delta \big) 
\, \widehat{u_{j_1+\ell_1}}(\tau,\eta)  &= \widehat{M\, u_{j_1,\cdot}} (\tau,\eta) 
+\dfrac{1-{\rm e}^{-\tau \, \Delta t}}{\tau} \, \sum_{\sigma=1}^{s+1} \sum_{\sigma'=0}^{\sigma-1} \sigma \, 
{\rm e}^{(\sigma-\sigma') \, \tau \, \Delta t} \, {\mathcal F}_{j_1}^{\sigma,\sigma'}(\eta) \, .\label{lem3'2}
\end{align}
where, in \eqref{lem3'1} and \eqref{lem3'2}, we have set
\begin{equation*}
{\mathcal F}_{j_1}^{\sigma,\sigma'}(\eta) =\sum_{\ell_1=-r_1}^{p_1} \left( \sum_{\ell'=-r'}^{p'} a_{\ell,\sigma} \, 
{\rm e}^{i\, \ell' \cdot \eta_\Delta} \, \right) \,  \widehat{f^{\sigma'}_{j_1+\ell_1,\cdot}} (\eta) \, ,
\end{equation*}
which corresponds to the partial Fourier transform with respect to $y=(x_2,\dots,x_d) \in \R^{d-1}$, of the step 
function associated with the sequence $(Q_\sigma \, f_j^{\sigma'})$ (no Laplace transform here).

We need to estimate integrals with respect to $(\theta,\eta)$ of the right hand side of \eqref{lem3'1} and 
\eqref{lem3'2}. The first term on the right of \eqref{lem3'1} and \eqref{lem3'2} are easy. For instance, we 
have (applying Plancherel Theorem):
\begin{align*}
\sum_{j_1 \in \Z} \int_{\R \times \R^{d-1}} \big| \widehat{L\, u_{j_1,\cdot}} (\tau,\eta) \big|^2 \, 
{\rm d}\theta \, {\rm d}\eta &=(2\, \pi)^d \, \sum_{j_1 \in \Z} \, \sum_{n \ge 0} \int_{n\, \Delta t}^{(n+1) \, \Delta t} 
{\rm e}^{-2\, \gamma \, s} \, \| L \, u_{j_1,\cdot}^n \|_{\ell^2(\Z^{d-1})}^2 \, {\rm d}s \\
&= (2\, \pi)^d \, \dfrac{1 -{\rm e}^{-2\, \gamma \, \Delta t}}{2\, \gamma \, \Delta t} \, \sum_{n \ge 0} 
\Delta t \, {\rm e}^{-2\, \gamma \, n \, \Delta t} \, \sum_{j_1 \in \Z} \| L \, u_{j_1,\cdot}^n \|_{\ell^2(\Z^{d-1})}^2 \, .
\end{align*}
We now recall that $\gamma \, \Delta t$ is restricted to the interval $(0,\ln R_0]$, and we use \eqref{estim1} 
to derive
\begin{equation*}
\sum_{j_1 \in \Z} \int_{\R \times \R^{d-1}} \big| \widehat{L\, u_{j_1,\cdot}} (\tau,\eta) \big|^2 \, 
{\rm d}\theta \, {\rm d}\eta \le C \, \left\{ \sum_{\sigma=0}^s \Ng f^\sigma \Nd_{-\infty,+\infty}^2 
+\sum_{n\ge s+1} \Delta t \, {\rm e}^{-2\, \gamma \, n\, \Delta t} \, \sum_{j_1 \le 0} 
\| g_{j_1,\cdot}^n \|_{\ell^2(\Z^{d-1})}^2 \right\} \, .
\end{equation*}
Similarly, we have
\begin{equation*}
\sum_{j_1 \le 0} \int_{\R \times \R^{d-1}} \big| \widehat{M\, u_{j_1,\cdot}} (\tau,\eta) \big|^2 \, 
{\rm d}\theta \, {\rm d}\eta =(2\, \pi)^d \, \dfrac{1 -{\rm e}^{-2\, \gamma \, \Delta t}}{2\, \gamma \, \Delta t} 
\, \sum_{n \ge 0} \Delta t \, {\rm e}^{-2\, \gamma \, n \, \Delta t} \, 
\sum_{j_1 \le 0} \| M \, u_{j_1,\cdot}^n \|_{\ell^2(\Z^{d-1})}^2 \, ,
\end{equation*}
which we can again uniformly estimate by the right hand side of \eqref{estimlem3'}.

Going back to the right hand side terms in \eqref{lem3'1} and \eqref{lem3'2}, we find that there only remains 
for proving \eqref{estimlem3'} to estimate the integral (here there are finitely many values of $\sigma$ and 
$\sigma'$):
\begin{equation*}
\sum_{j_1 \in \Z} \int_{\R \times \R^{d-1}} \left| \dfrac{1-{\rm e}^{-\tau \, \Delta t}}{\tau} \right|^2 \, 
\big| {\mathcal F}_{j_1}^{\sigma,\sigma'}(\eta) \big|^2 \, {\rm d}\theta \, {\rm d}\eta 
=\left( \int_\R \left| \dfrac{1-{\rm e}^{-\tau \, \Delta t}}{\tau} \right|^2 \, {\rm d}\theta \right) \, 
\sum_{j_1 \in \Z} \int_{\R^{d-1}} \big| {\mathcal F}_{j_1}^{\sigma,\sigma'}(\eta) \big|^2 \, {\rm d}\eta \, ,
\end{equation*}
where we have applied Fubini Theorem. Applying first Plancherel Theorem with respect to the $d-1$ last space 
variables, we get
\begin{equation*}
\sum_{j_1 \in \Z} \int_{\R^{d-1}} \big| {\mathcal F}_{j_1}^{\sigma,\sigma'}(\eta) \big|^2 \, {\rm d}\eta \le C \, 
\sum_{j_1 \in \Z} \sum_{j' \in \Z^{d-1}} \left( \prod_{k=2}^d \Delta x_k \right) \, (f_j^{\sigma'})^2 \le 
\dfrac{C}{\Delta t} \, \sum_{\sigma=0}^s \Ng f^\sigma \Nd_{-\infty,+\infty}^2 \, .
\end{equation*}
The conclusion then follows by computing
\begin{equation*}
\int_\R \left| \dfrac{1 -{\rm e}^{-\tau \, \Delta t}}{\tau} \right|^2 \, {\rm d}\theta =
2\, \pi \, \Delta t \, \dfrac{1 -{\rm e}^{-2 \, \gamma \, \Delta t}}{2 \, \gamma \, \Delta t} \, ,
\end{equation*}
and by recalling that $\gamma \, \Delta t$ belongs to $(0,\ln R_0]$. We can eventually bound the integrals on 
the left hand side of \eqref{estimlem3'} by estimating separately the integrals of each term on the right hand side 
of \eqref{lem3'1} and \eqref{lem3'2}.
\end{proof}

\noindent The conclusion now relies on the following crucial result.

\begin{lemma}[The trace estimate]
\label{lem3}
Let Assumptions \ref{assumption0}, \ref{assumption1} and \ref{assumption2} be satisfied. Let $R_0>1$ be fixed 
as above and let $P_1 \in \N$. Then there exists a constant $C_{P_1}>0$ such that for all $z \in \U$ with $|z| \le 
R_0$, for all $\eta \in \R^{d-1}$ and for all sequence $(w_{j_1})_{j_1 \in \Z} \in \ell^2(\Z;\C)$, there holds
\begin{equation}
\label{estimlem3}
\sum_{j_1=-r_1-p_1}^{P_1} |w_{j_1}|^2 \le C_{P_1} \, \left\{ \sum_{j_1 \in \Z} \left| \sum_{\ell_1=-r_1}^{p_1} 
a_{\ell_1}(z,\eta) \, w_{j_1+\ell_1} \right|^2 
+\sum_{j_1 \le 0} \left| \sum_{\ell_1=-r_1}^{p_1} z \, \partial_z a_{\ell_1}(z,\eta) \, w_{j_1+\ell_1} \right|^2 
\right\} \, .
\end{equation}
Recall that the functions $a_{\ell_1}$, $\ell_1=-r_1,\dots,p_1$, are defined in \eqref{defA-d}.
\end{lemma}

The proof of Lemma \ref{lem3} is rather long. Before giving it in full details, we indicate how Lemma \ref{lem3} 
yields the result of Theorem \ref{absorbing}. We apply Lemma \ref{lem3} to $z=\exp (\tau \, \Delta t)$, 
$\tau=\gamma +i\, \theta$ with $\gamma \, \Delta t \in (0, \ln R_0]$, $\eta_\Delta \in \R^{d-1}$ and 
the sequence $(\widehat{u_{j_1}} (\tau,\eta))_{j_1 \in \Z}$. We then integrate \eqref{estimlem3} with respect to 
$(\theta,\eta)$ and use Lemma \ref{lem3'} to derive
\begin{equation*}
\sum_{j_1=-r_1-p_1}^{P_1} \int_{\R \times \R^{d-1}} \left| \widehat{u_{j_1}} (\gamma+i\, \theta,\eta) \right|^2 \, 
{\rm d}\theta \, {\rm d}\eta \le C \, \left\{ \sum_{\sigma=0}^s \Ng f^\sigma \Nd_{-\infty,+\infty}^2 
+\sum_{n\ge s+1} \Delta t \, {\rm e}^{-2\, \gamma \, n\, \Delta t} \, \sum_{j_1 \le 0} 
\| g_{j_1,\cdot}^n \|_{\ell^2(\Z^{d-1})}^2 \right\} .
\end{equation*}
It remains to apply Plancherel Theorem and we get
\begin{multline*}
\sum_{j_1=-r_1-p_1}^{P_1} \sum_{n \in \N} \dfrac{1-{\rm e}^{-2\, \gamma \, \Delta t}}{2\, \gamma \, \Delta t} \, 
\Delta t \, {\rm e}^{-2\, \gamma \, n\, \Delta t} \, \| u_{j_1,\cdot}^n \|_{\ell^2(\Z^{d-1})}^2 \\
\le C \, \left\{ \sum_{\sigma=0}^s \Ng f^\sigma \Nd_{-\infty,+\infty}^2 
+\sum_{n\ge s+1} \Delta t \, {\rm e}^{-2\, \gamma \, n\, \Delta t} \, \sum_{j_1 \le 0} 
\| g_{j_1,\cdot}^n \|_{\ell^2(\Z^{d-1})}^2 \right\} \, .
\end{multline*}
Recalling that $\gamma \, \Delta t$ is restricted to the interval $(0, \ln R_0]$, we have thus derived the trace 
estimate
\begin{equation*}
\sum_{n \in \N} \Delta t \, {\rm e}^{-2\, \gamma \, n\, \Delta t} \, \sum_{j_1=-r_1-p_1}^{P_1} 
\| u_{j_1,\cdot}^n \|_{\ell^2(\Z^{d-1})}^2 
\le C \, \left\{ \sum_{\sigma=0}^s \Ng f^\sigma \Nd_{-\infty,+\infty}^2 
+\sum_{n\ge s+1} \Delta t \, {\rm e}^{-2\, \gamma \, n\, \Delta t} \, \sum_{j_1 \le 0} 
\| g_{j_1,\cdot}^n \|_{\ell^2(\Z^{d-1})}^2 \right\} \, .
\end{equation*}
Combined with the semigroup and interior estimate \eqref{estim1}, this gives the estimate \eqref{estimabsorb} 
of Theorem \ref{absorbing} for $\gamma \, \Delta t \in (0,\ln R_0]$.
\end{proof}

\begin{proof}[Proof of Lemma \ref{lem3}]
Let us recall that the functions $a_{\ell_1}$ are $2\, \pi$-periodic with respect to each coordinate of $\eta$. 
We can therefore restrict to $\eta \in [0,2\, \pi]^{d-1}$ rather than considering $\eta \in \R^{d-1}$. We argue 
by contradiction and assume that the conclusion to Lemma \ref{lem3} does not hold. This means the following, 
up to normalizing and extracting subsequences; there exist three sequences (indexed by $k \in \N$):
\begin{itemize}
 \item a sequence $(w^k)_{k \in \N}$ with values in $\ell^2(\Z;\C)$ such that $(w_{-r_1-p_1}^k,\dots,w_{P_1}^k)$ 
 belongs to the unit sphere of $\C^{P_1+r_1+p_1+1}$ for all $k$, and $(w_{-r_1-p_1}^k,\dots,w_{P_1}^k)$ 
 converges towards $(\underline{w}_{-r_1-p_1},\dots,\underline{w}_{P_1})$ as $k$ tends to infinity,
 
 \item a sequence $(z^k)_{k \in \N}$ with values in $\U \cap \{ \zeta \in \C \, , \, |\zeta| \le R_0 \}$, which converges 
 towards $\underline{z} \in \Ubar$,
 
 \item a sequence $(\eta^k)_{k \in \N}$ with values in $[0,2\, \pi]^{d-1}$, which converges towards $\underline{\eta} 
 \in [0,2\, \pi]^{d-1}$,
\end{itemize}
and these sequences satisfy:
\begin{equation}
\label{lem3-1}
\lim_{k \rightarrow +\infty} \quad \sum_{j_1 \in \Z} \left| \sum_{\ell_1=-r_1}^{p_1} 
a_{\ell_1}(z^k,\eta^k) \, w^k_{j_1+\ell_1} \right|^2 
+\sum_{j_1 \le 0} \left| \sum_{\ell_1=-r_1}^{p_1} z^k \, \partial_z a_{\ell_1}(z^k,\eta^k) \, w^k_{j_1+\ell_1} 
\right|^2 =0 \, .
\end{equation}
We are going to show that \eqref{lem3-1} implies that $(\underline{w}_{-r_1-p_1},\dots,\underline{w}_{P_1})$ 
must be zero, which will yield a contradiction since this vector must have norm $1$.

$\bullet$ Let us first show that each component $(w^k_{j_1})_{k \in \N}$, $j_1 \in \Z$, has a limit as $k$ tends 
to infinity. This is already clear for $j_1=-r_1-p_1,\dots,P_1$. For $j_1>P_1$, we argue by induction. From 
\eqref{lem3-1}, we have
\begin{equation*}
\lim_{k \rightarrow +\infty} \quad \sum_{\ell_1=-r_1}^{p_1} a_{\ell_1}(z^k,\eta^k) \, w^k_{P_1-p_1+1+\ell_1} =0 \, ,
\end{equation*}
and by Assumption \ref{assumption2}, we know that $a_{p_1}(\underline{z},\underline{\eta})$ is nonzero. 
Hence $(w^k_{P_1+1})_{k \in \N}$ converges towards
\begin{equation*}
-\dfrac{1}{a_{p_1}(\underline{z},\underline{\eta})} \, \sum_{\ell_1=-r_1}^{p_1-1} a_{\ell_1}(\underline{z},\underline{\eta}) 
\, \underline{w}_{P_1-p_1+1+\ell_1} \, ,
\end{equation*}
which we define as $\underline{w}_{P_1+1}$. We can argue by induction in the same way for all indices 
$j_1>P_1+1$, but also for indices $j_1<-r_1-p_1$ because the function $a_{-r_1}$ also does not vanish 
on $\Ubar \times \R^{d-1}$.

Using \eqref{lem3-1}, we have thus shown that for each $j_1 \in \Z$, $(w^k_{j_1})_{k \in \N}$ tends towards 
some limit $\underline{w}_{j_1}$ as $k$ tends to infinity, and the sequence $\underline{w}$, which does not 
necessarily belong to $\ell^2(\Z;\C)$, satisfies the induction relations:
\begin{align}
\forall \, j_1 \in \Z \, ,\quad & \sum_{\ell_1=-r_1}^{p_1} a_{\ell_1}(\underline{z},\underline{\eta}) \, 
\underline{w}_{j_1+\ell_1} =0 \, ,\label{induction1}\\
\forall \, j_1 \le 0 \, ,\quad & \sum_{\ell_1=-r_1}^{p_1} \underline{z} \, \partial_z a_{\ell_1}(\underline{z},\underline{\eta}) 
\, \underline{w}_{j_1+\ell_1} =0 \, .\label{induction2}
\end{align}

$\bullet$ The induction relation \eqref{induction1} is the one that arises in \cite{gks,michelson} and all the works that 
deal with strong stability. The main novelty here is to use simultaneously \eqref{induction1} for controlling the unstable 
components of $(\underline{w}_{-r_1-p_1},\dots,\underline{w}_{-1})$ and \eqref{induction2} for controlling the stable 
components of $(\underline{w}_{-r_1-p_1},\dots,\underline{w}_{-1})$. The fact that $\underline{w}$ satisfies 
simultaneously \eqref{induction1} and \eqref{induction2} for $j_1 \le 0$ automatically annihilates the central components. 
This sketch of proof is made precise below.

We define the source terms:
\begin{equation*}
F_{j_1}^k := \sum_{\ell_1=-r_1}^{p_1} a_{\ell_1}(z^k,\eta^k) \, w^k_{j_1+\ell_1} \, ,\quad 
G_{j_1}^k := \sum_{\ell_1=-r_1}^{p_1} z^k \, \partial_z a_{\ell_1}(z^k,\eta^k) \, w^k_{j_1+\ell_1} \, ,
\end{equation*}
which, according to \eqref{lem3-1}, satisfy
\begin{equation}
\label{lem3-2}
\lim_{k \rightarrow 0} \quad \sum_{j_1 \in \Z} |F_{j_1}^k|^2 =0 \, ,\quad 
\lim_{k \rightarrow 0} \quad \sum_{j_1 \le 0} |G_{j_1}^k|^2 =0 \, .
\end{equation}
We also introduce the vectors (here $T$ denotes transposition)
\begin{equation*}
W_{j_1}^k := \Big( w^k_{j_1+p_1},\dots,w^k_{j_1+1-r_1} \Big)^T \, ,\quad 
\underline{W}_{j_1} := \Big( \underline{w}_{j_1+p_1},\dots,\underline{w}_{j_1+1-r_1} \Big)^T \, ,
\end{equation*}
and the matrices:
\begin{align}
\LL (z,\eta) &:= \begin{pmatrix}
-a_{p_1-1} (z,\eta)/a_{p_1} (z,\eta) & \dots & \dots & -a_{-r_1} (z,\eta)/a_{p_1} (z,\eta) \\
1 & 0 & \dots & 0 \\
0 & \ddots & \ddots & \vdots \\
0 & 0 & 1 & 0 \end{pmatrix} \in {\mathcal M}_{p_1+r_1}(\C) \, ,\label{defL} \\
\M (z,\eta) &:= \begin{pmatrix}
-\partial_z a_{p_1-1} (z,\eta)/\partial_z a_{p_1} (z,\eta) & \dots & \dots & 
-\partial_z a_{-r_1} (z,\eta)/\partial_z a_{p_1} (z,\eta) \\
1 & 0 & \dots & 0 \\
0 & \ddots & \ddots & \vdots \\
0 & 0 & 1 & 0 \end{pmatrix} \in {\mathcal M}_{p_1+r_1}(\C) \, .\label{defM}
\end{align}
The matrix $\LL$ is well-defined on $\Ubar \times \R^{d-1}$ according to Assumption \ref{assumption2}. The 
matrix $\M$ is also well-defined on $\Ubar \times \R^{d-1}$ because for any $\eta \in \R^{d-1}$, Assumption 
\ref{assumption2} asserts that $a_{p_1}(\cdot,\eta)$ is a nonconstant polynomial whose roots lie in $\D$. From 
the Gauss-Lucas Theorem, the roots of $\partial_z a_{p_1}(\cdot,\eta)$ lie in the convex hull of those of 
$a_{p_1}(\cdot,\eta)$. Therefore $\partial_z a_{p_1}(\cdot,\eta)$ does not vanish on $\Ubar$. In the same 
way, $\partial_z a_{-r_1}(\cdot,\eta)$ does not vanish on $\Ubar$.

With our above notation, the vectors $W_{j_1}^k$, $\underline{W}_{j_1}$, satisfy the one step induction relations:
\begin{align}
\forall \, j_1 \in \Z \, ,\quad W_{j_1+1}^k &= \LL(z^k,\eta^k) \, W_{j_1}^k 
+\Big( F^k_{j_1+1}/a_{p_1} (z^k,\eta^k),0,\dots,0 \Big)^T \, ,\quad 
\underline{W}_{j_1+1} =\LL(\underline{z},\underline{\eta}) \, \underline{W}_{j_1} \, ,\label{induction1'} \\
\forall \, j_1 \le -1 \, ,\quad W_{j_1+1}^k &= \M(z^k,\eta^k) \, W_{j_1}^k 
+\Big( G^k_{j_1+1}/(z^k \, \partial_z a_{p_1} (z^k,\eta^k)),0,\dots,0 \Big)^T \, ,\quad 
\underline{W}_{j_1+1} =\M(\underline{z},\underline{\eta}) \, \underline{W}_{j_1} \, .\label{induction2'}
\end{align}

$\bullet$ From Assumption \ref{assumption2} and the above application of the Gauss-Lucas Theorem, we 
already know that both matrices $\LL(z,\eta)$ and $\M(z,\eta)$ are invertible for $(z,\eta) \in \Ubar \times 
\R^{d-1}$. Furthermore, Assumption \ref{assumption1} shows that $\LL(z,\eta)$ has no eigenvalue on 
$\cercle$ for $(z,\eta) \in \U \times \R^{d-1}$. This property dates back at least to \cite{kreiss1}. However, 
central eigenvalues on $\cercle$ may occur for $\LL$ when $z$ belongs to $\cercle$. The crucial point for 
proving Lemma \ref{lem3} is that Assumption \ref{assumption1} precludes central eigenvalues of $\M$ for 
all $z \in \Ubar$. Namely, for all $z \in \Ubar$ and all $\eta \in \R^{d-1}$, $\M(z,\eta)$ has no eigenvalue on 
$\cercle$. This property holds because otherwise, for some $(z,\eta) \in \Ubar \times \R^{d-1}$, there would 
exist a solution $\kappa_1 \in \cercle$ to the dispersion relation
\begin{equation*}
\sum_{\ell_1=-r_1}^{p_1} z \, \partial_z a_{\ell_1}(z,\eta) \, \kappa_1^{\ell_1} =0 \, .
\end{equation*}
For convenience, the coordinates of $\eta$ are denoted $(\eta_2,\dots,\eta_d)$. Using the definition \eqref{defA-d} 
of $a_{\ell_1}$, and defining $\kappa :=(\kappa_1,{\rm e}^{i\, \eta_2},\dots,{\rm e}^{i\, \eta_d})$, we have found a 
root $z \in \Ubar$ to the relation
\begin{equation}
\label{dispersionder}
\sum_{\sigma=1}^{s+1} \sigma \, \widehat{Q_\sigma} (\kappa) \, z^{\sigma-1} =0 \, ,
\end{equation}
but this is not possible because the $s+1$ roots (in $z$) to the dispersion relation \eqref{dispersion} are simple 
and belong to $\Dbar$. The Gauss-Lucas Theorem thus shows that the roots to the relation \eqref{dispersionder} 
belong to $\D$ (and therefore not to $\Ubar$).

At this stage, we know that the eigenvalues of $\M(z,\eta)$, $(z,\eta) \in \Ubar \times \R^{d-1}$, split into two 
groups: those in $\U$, which we call the unstable ones, and those in $\D$, which we call the stable ones. For 
$(z,\eta) \in \Ubar \times \R^{d-1}$, we then introduce the spectral projector $\Pi_\M^s(z,\eta)$, resp. $\Pi_\M^u(z,\eta)$, 
of $\M(z,\eta)$ on the generalized eigenspace associated with eigenvalues in $\D$, resp. $\U$. We can then integrate 
the first induction relation in \eqref{induction2'} and get
\begin{equation*}
\Pi_\M^s(z^k,\eta^k) \, W_0^k =\dfrac{1}{z^k \, \partial_z a_{p_1} (z^k,\eta^k)} \, \sum_{j_1 \le 0} 
\M(z^k,\eta^k)^{|j_1|} \, \Pi_\M^s(z^k,\eta^k) \, \Big( G^k_{j_1},0,\dots,0 \Big)^T \, .
\end{equation*}
The projector $\Pi_\M^s$ depends continuously on $(z,\eta) \in \Ubar \times \R^{d-1}$. Furthermore, since the 
spectrum of $\M$ does not meet $\cercle$ even for $z \in \cercle$, there exists a constant $C>0$ and a $\delta \in 
(0,1)$ that are independent of $k \in \N$ and such that
\begin{equation*}
\forall \, j_1 \le 0 \, ,\quad \| \M(z^k,\eta^k)^{|j_1|} \, \Pi_\M^s(z^k,\eta^k) \| \le C \, \delta^{|j_1|} \, .
\end{equation*}
We thus get a uniform estimate with respect to $k$:
\begin{equation*}
|\Pi_\M^s(z^k,\eta^k) \, W_0^k|^2 \le C \, \sum_{j_1 \le 0} |G^k_{j_1}|^2 \, .
\end{equation*}
Passing to the limit and using \eqref{lem3-2}, we get $\Pi_\M^s (\underline{z},\underline{\eta}) \, \underline{W}_0 =0$, 
or in other words $\underline{W}_0 =\Pi_\M^u (\underline{z},\underline{\eta}) \, \underline{W}_0$.

$\bullet$ The sequence $(\underline{W}_{j_1})_{j_1 \le 0}$ satisfies both induction relations \eqref{induction1'} and 
\eqref{induction2'}. Due to the form of the companion matrices $\LL$ and $\M$, see \eqref{defL}-\eqref{defM}, 
we can conclude that the vector $\underline{W}_0$ belongs to the generalized eigenspace (of either $\LL$ or 
$\M$) associated with the common eigenvalues of $\M(\underline{z},\underline{\eta})$ and $\LL(\underline{z},
\underline{\eta})$. We have already seen that $\M(\underline{z},\underline{\eta})$ has no eigenvalue on 
$\cercle$ and $\underline{W}_0 =\Pi_\M^u (\underline{z},\underline{\eta}) \, \underline{W}_0$, so we can 
conclude that $\underline{W}_0$ belongs to the generalized eigenspace of $\LL$ associated with those 
common eigenvalues of $\M(\underline{z},\underline{\eta})$ and $\LL(\underline{z},\underline{\eta})$ in 
$\U$.

The matrix $\LL(\underline{z},\underline{\eta})$ has $N^u$ eigenvalues in $\U$, $N^s$ in $\D$ and $N^c$ on 
$\cercle$. (Since $\underline{z}$ may belong to $\cercle$, $N^c$ is not necessarily zero.) With obvious notations, 
we let $\Pi_\LL^{u,s,c}(z,\eta)$ denote the corresponding spectral projectors of $\LL$ for $(z,\eta)$ sufficiently 
close to $(\underline{z},\underline{\eta})$. In particular, the eigenvalues corresponding to $\Pi_\LL^u(z,\eta)$ 
lie in $\U$ uniformly away from $\cercle$ for $(z,\eta)$ sufficiently close to $(\underline{z},\underline{\eta})$. 
We can then integrate the first induction relation in \eqref{induction1'} and derive (for $k$ sufficiently large):
\begin{equation*}
\Pi_\LL^u(z^k,\eta^k) \, W_0^k =-\dfrac{1}{a_{p_1} (z^k,\eta^k)} \, \sum_{j_1 \ge 0} 
\LL(z^k,\eta^k)^{-j_1-1} \, \Pi_\LL^u(z^k,\eta^k) \, \Big( F^k_{j_1},0,\dots,0 \Big)^T \, .
\end{equation*}
Using the uniform exponential decay of $\LL(z^k,\eta^k)^{-j_1-1} \, \Pi_\LL^u(z^k,\eta^k)$ and \eqref{lem3-2}, 
we finally end up with 
\begin{equation*}
\Pi_\LL^u(\underline{z},\underline{\eta}) \, \underline{W}_0=0 \, .
\end{equation*}
Since $\underline{W}_0$ belongs to the generalized eigenspace of $\LL$ associated with those common 
eigenvalues of $\M(\underline{z},\underline{\eta})$ and $\LL(\underline{z},\underline{\eta})$ in $\U$, we can 
conclude that $\underline{W}_0$ equals zero. Applying the induction relation \eqref{induction1'}, the whole 
sequence $(\underline{W}_{j_1})_{j_1 \in \Z}$ is zero which yields the expected contradiction.
\end{proof}

The crucial property that we use in the proof of Lemma \ref{lem3} is the fact that up to $z \in \cercle$, the 
eigenvalues of $\M(z,\eta)$ lie either in $\D$ or $\U$. For the leap-frog scheme, this property would not be 
true if we had imposed the auxiliary numerical boundary condition $u_j^{n+2}+u_j^n$ rather than $2\, u_j^{n+2} 
+\lambda \, a \, (u_{j+1}^{n+1}-u_{j-1}^{n+1})$.

Let us also observe that we have used the fact that $a_{p_1}$ and $a_{-r_1}$ are nonconstant in order to 
study the induction relation \eqref{induction2}. There might be some schemes for which $a_{p_1}$ and/or 
$a_{-r_1}$ are constant but for which one can still apply similar arguments as in the previous proof, even 
though \eqref{induction2} is an induction relation with fewer steps than \eqref{induction1}. In this respect, 
Assumption \ref{assumption2} might be relaxed in specific applications.

\begin{remark}
The auxiliary problem \eqref{numabsorbing} is in general not of the same form as \eqref{numibvp} because 
in \eqref{numabsorbing} one has to impose infinitely many numerical boundary conditions. This is due to the 
fact that the stencil of $M$ incorporates points \lq \lq on the left\rq \rq $\,$ with respect to the first space 
variable. A remarkable exception occurs for explicit schemes with $s=0$, for in that case the multiplier 
$M \, v_j^n$ reads $v_j^{n+1}$ and \eqref{numabsorbing} is exactly the auxiliary problem considered 
(and labeled (2.7)) in \cite{jfcag} where one imposes Dirichlet boundary conditions on finitely many 
boundary meshes (just use $g_j^n =0$ for $j_1 \le -r_1$). The reader can then check that the energy-dissipation 
balance law of Proposition \ref{prop2} in that case ($s=0$, $Q_1 =I$) coincides exactly with the algebra involved 
in the derivation of the estimate (2.12) in \cite{jfcag}. The reader can also check that for $s=0$, and $Q_1=I$, 
Lemma \ref{lem3} becomes a rather trivial exercise...

There still remains the problem of constructing a set of dissipative numerical boundary conditions of the same 
form as \eqref{numibvp} with $s \ge 1$, that is with finitely many numerical boundary conditions, and for which 
one can prove by hand both a semigroup and a trace estimate as in Theorem \ref{absorbing}.
\end{remark}

\subsection{End of the proof}

As explained in the introduction of Section \ref{section3}, the linearity of \eqref{numibvp} reduces the proof of 
Theorem \ref{mainthm} to the case $(F_j^n)=0$, $(g_j^n)=0$, since we have already dealt with the case of zero 
initial data. We thus focus on \eqref{numibvp} with $(F_j^n)=0$ and $(g_j^n)=0$, and write the corresponding 
solution $(u_j^n)$ as $u_j^n =v_j^n +w_j^n$, where the sequence $(v_j^n)$ solves:
\begin{equation}
\label{defvjn}
\begin{cases}
L \, v_j^n =0 \, ,& j \in \Z^d \, ,\quad j_1 \ge 1\, ,\quad n\ge 0 \, ,\\
M \, v_j^n =0 \, ,& j \in \Z^d \, ,\quad j_1 \le 0\, ,\quad n\ge 0 \, ,\\
v_j^n = f_j^n \, ,& j \in \Z^d \, ,\quad n=0,\dots,s \, ,
\end{cases}
\end{equation}
and $(w_j^n)$ solves:
\begin{equation}
\label{defwjn}
\begin{cases}
L \, w_j^n =0 \, ,& j \in \Z^d \, ,\quad j_1 \ge 1\, ,\quad n\ge 0 \, ,\\
w_j^{n+s+1} +{\dps \sum_{\sigma=0}^{s+1}} B_{j_1,\sigma} \, w_{1,j'}^{n+\sigma} =\tilde{g}_j^{n+s+1} \, ,& 
j \in \Z^d \, ,\quad j_1=1-r_1,\dots,0\, ,\quad n\ge 0 \, ,\\
w_j^n = 0 \, ,& j \in \Z^d \, ,\quad n=0,\dots,s \, .
\end{cases}
\end{equation}
For $v_j^n +w_j^n$ to coincide with the solution $(u_j^n)$ to \eqref{numibvp}, it is sufficient to extend the initial 
data $f_j^0,\dots,f_j^s$ by zero for $j_1 \le -r_1$, which provides with the initial data in \eqref{defvjn} on all $\Z^d$, 
and to define the boundary source term in \eqref{defwjn} by:
\begin{equation}
\label{defgtilde}
\tilde{g}_j^{n+s+1} := -v_j^{n+s+1} -{\dps \sum_{\sigma=0}^{s+1}} B_{j_1,\sigma} \, v_{1,j'}^{n+\sigma} \, .
\end{equation}

We can estimate the solution $(v_j^n)$ to \eqref{defvjn} by applying Theorem \ref{absorbing}. In particular, 
the trace estimate:
\begin{equation*}
\sum_{n\ge 0} \Delta t \, {\rm e}^{-2\, \gamma \, n\, \Delta t} \, \sum_{j_1=1-r_1}^{P_1} 
\| v_{j_1,\cdot}^n \|_{\ell^2(\Z^{d-1})}^2 \le C \, \sum_{\sigma=0}^s \Ng f^\sigma \Nd_{1-r_1,+\infty}^2 \, ,
\end{equation*}
for $P_1=\max(p_1,q_1+1)$ gives (recall the definition \eqref{defgtilde} of $\tilde{g}_j^{n+s+1}$):
\begin{align*}
\sum_{n\ge s+1} \Delta t \, {\rm e}^{-2\, \gamma \, n\, \Delta t} \, \sum_{j_1=1-r_1}^0 
\| \tilde{g}_{j_1,\cdot}^n \|_{\ell^2(\Z^{d-1})}^2 
&\le C \, \sum_{n\ge 0} \Delta t \, {\rm e}^{-2\, \gamma \, n\, \Delta t} \, \sum_{j_1=1-r_1}^{\max(p_1,q_1+1)} 
\| v_{j_1,\cdot}^n \|_{\ell^2(\Z^{d-1})}^2 \\
&\le C \, \sum_{\sigma=0}^s \Ng f^\sigma \Nd_{1-r_1,+\infty}^2 \, .
\end{align*}
We can apply Theorem \ref{mainthm} to the solution $(w_j^n)$ to \eqref{defwjn} because the initial data 
in \eqref{defwjn} vanish. We get:
\begin{multline*}
\sup_{n \ge 0} \, {\rm e}^{-2\, \gamma \, n\, \Delta t} \, \Ng w^n \Nd_{1-r_1,+\infty}^2 
+\dfrac{\gamma}{\gamma \, \Delta t+1} \, 
\sum_{n\ge 0} \Delta t \, {\rm e}^{-2\, \gamma \, n\, \Delta t} \, \Ng w^n \Nd_{1-r_1,+\infty}^2 \\
+\sum_{n\ge 0} \Delta t \, {\rm e}^{-2\, \gamma \, n\, \Delta t} \, \sum_{j_1=1-r_1}^{p_1} 
\| w_{j_1,\cdot}^n \|_{\ell^2(\Z^{d-1})}^2 \le C \, \sum_{n\ge s+1} \Delta t \, {\rm e}^{-2\, \gamma \, n\, \Delta t} 
\, \sum_{j_1=1-r_1}^0 \| \tilde{g}_{j_1,\cdot}^n \|_{\ell^2(\Z^{d-1})}^2 \\
\le C \, \sum_{\sigma=0}^s \Ng f^\sigma \Nd_{1-r_1,+\infty}^2 \, .
\end{multline*}
Combining with the similar estimate provided by Theorem \ref{absorbing} for $(v_j^n)$, we end up with 
the expected estimate:
\begin{multline*}
\sup_{n \ge 0} \, {\rm e}^{-2\, \gamma \, n\, \Delta t} \, \Ng u^n \Nd_{1-r_1,+\infty}^2 
+\dfrac{\gamma}{\gamma \, \Delta t+1} \, 
\sum_{n\ge 0} \Delta t \, {\rm e}^{-2\, \gamma \, n\, \Delta t} \, \Ng u^n \Nd_{1-r_1,+\infty}^2 \\
+\sum_{n\ge 0} \Delta t \, {\rm e}^{-2\, \gamma \, n\, \Delta t} \, \sum_{j_1=1-r_1}^{p_1} 
\| u_{j_1,\cdot}^n \|_{\ell^2(\Z^{d-1})}^2 \le C \, \sum_{\sigma=0}^s \Ng f^\sigma \Nd_{1-r_1,+\infty}^2 \, ,
\end{multline*}
which completes the proof of Theorem \ref{mainthm}.

\section{Conclusion and perspectives}

Let us first observe that in \cite{wade}, {\sc Wade} has constructed symmetrizers for deriving stability 
estimates for multistep schemes, even in the case of variable coefficients. His conditions for constructing 
a symmetrizer are less restrictive than Assumption \ref{assumption1}. However, the symmetrizer in 
\cite{wade} is genuinely nonlocal and it is therefore not clear that it may be useful for boundary value 
problems. The main novelty here is to construct a {\it local} multiplier whose properties allow for the 
design of an auxiliary {\it dissipative} boundary value problem. This is our key to Theorem \ref{mainthm}.
\bigskip

In this article we have always discarded the dissipation term provided by the nonnegative form $D_0$. 
For the approximation of parabolic equations, this term may give some extra dissipation, but a crucial point 
to keep in mind is that the coefficients of the numerical scheme are assumed to be constant (which may in 
turn yield rather severe CFL conditions for implicit approximations of parabolic equations). Hence it does 
not seem very clear that our approach will yield stability estimates with \lq \lq optimal\rq \rq $\,$ CFL 
conditions when approximating parabolic equations. This extension is left to further study in the future.
\bigskip

The main possible improvement of Theorem \ref{mainthm} would consist of assuming that only the roots to 
\eqref{dispersion} that lie on $\cercle$ are simple. Here we have assumed that all the roots, including those 
in $\D$ are simple. If we could manage to deal with multiple roots in $\D$, then Theorem \ref{mainthm} would 
be applicable to numerical approximations of the transport equation \eqref{transport} that are based on 
Adams-Bashforth methods of order $3$ or higher (such methods have $0$ as a root of multiplicity $2$ or 
more at the zero frequency).
\bigskip

The results in this paper achieve the proof of a "weak form" of the conjecture in \cite{kreiss-wu} that strong 
stability, in the sense of Definition \ref{defstab1}, implies semigroup stability. However, an even stronger 
assumption was made in \cite{kreiss-wu}, namely that the sole fulfillment of the interior estimate
\begin{equation*}
\dfrac{\gamma}{\gamma \, \Delta t+1} \, \sum_{n\ge s+1} \Delta t \, {\rm e}^{-2\, \gamma \, n\, \Delta t} \, 
\Ng u^n \Nd_{1-r_1,+\infty}^2 \le C \, \dfrac{\gamma \, \Delta t+1}{\gamma} \, \sum_{n\ge s+1} \, \Delta t \, 
{\rm e}^{-2\, \gamma \, n\, \Delta t} \, \Ng F^n \Nd_{1,+\infty}^2 \, ,
\end{equation*}
when both the initial {\it and boundary data} for \eqref{numibvp} vanish, does imply semigroup stability. 
The analogous conjecture for partial differential equations seems to be still open so far, but we do hope 
that our multiplier technique may yield some insight for dealing with the strong form of the conjecture in 
\cite{kreiss-wu}.

\paragraph{Acknowledgments} This article was completed while the author was visiting the Institute 
of Mathematical Science at Nanjing University. The author warmly thanks Professor Yin Huicheng 
and the Institute for their hospitality during this visit.

\bibliographystyle{alpha}
\bibliography{LG}

\begin{thebibliography}{HNW93}

\bibitem[AG76]{ag1}
S.~Abarbanel and D.~Gottlieb.
\newblock A note on the leap-frog scheme in two and three space dimensions.
\newblock {\em J. Computational Phys.}, 21(3):351--355, 1976.

\bibitem[AG79]{ag2}
S.~Abarbanel and D.~Gottlieb.
\newblock Stability of two-dimensional initial boundary value problems using
  leap-frog type schemes.
\newblock {\em Math. Comp.}, 33(148):1145--1155, 1979.

\bibitem[BGS07]{benzoni-serre}
S.~Benzoni-Gavage and D.~Serre.
\newblock {\em Multidimensional hyperbolic partial differential equations}.
\newblock Oxford University Press, 2007.
\newblock First-order systems and applications.

\bibitem[CG11]{jfcag}
J.-F. Coulombel and A.~Gloria.
\newblock Semigroup stability of finite difference schemes for multidimensional
  hyperbolic initial boundary value problems.
\newblock {\em Math. Comp.}, 80(273):165--203, 2011.

\bibitem[Cou09]{jfcsinum}
J.-F. Coulombel.
\newblock Stability of finite difference schemes for hyperbolic initial
  boundary value problems.
\newblock {\em SIAM J. Numer. Anal.}, 47(4):2844--2871, 2009.

\bibitem[Cou13]{jfcnotes}
J.-F. Coulombel.
\newblock Stability of finite difference schemes for hyperbolic initial
  boundary value problems.
\newblock In {\em HCDTE Lecture Notes. Part I. Nonlinear Hyperbolic PDEs,
  Dispersive and Transport Equations}, pages 97--225. American Institute of
  Mathematical Sciences, 2013.

\bibitem[Cou14]{jfc}
J.-F. Coulombel.
\newblock Fully discrete hyperbolic initial boundary value problems with
  nonzero initial data.
\newblock {\em Preprint}, {\tt http://arxiv.org/abs/1412.0851}, 2014.

\bibitem[G{\aa}r56]{garding}
L.~G{\aa}rding.
\newblock Solution directe du probl\`eme de {C}auchy pour les \'equations
  hyperboliques.
\newblock In {\em La th\'eorie des \'equations aux d\'eriv\'ees partielles},
  Colloques Internationaux du C. N. R. S., pages 71--90. C. N. R. S., Paris,
  1956.

\bibitem[GKO95]{gko}
B.~Gustafsson, H.-O. Kreiss, and J.~Oliger.
\newblock {\em Time dependent problems and difference methods}.
\newblock John Wiley \& Sons, 1995.

\bibitem[GKS72]{gks}
B.~Gustafsson, H.-O. Kreiss, and A.~Sundstr{\"o}m.
\newblock Stability theory of difference approximations for mixed initial
  boundary value problems. {II}.
\newblock {\em Math. Comp.}, 26(119):649--686, 1972.

\bibitem[GT81]{goldberg-tadmor}
M.~Goldberg and E.~Tadmor.
\newblock Scheme-independent stability criteria for difference approximations
  of hyperbolic initial-boundary value problems. {II}.
\newblock {\em Math. Comp.}, 36(154):603--626, 1981.

\bibitem[HNW93]{hnw}
E.~Hairer, S.~P. N{\o}rsett, and G.~Wanner.
\newblock {\em Solving ordinary differential equations. {I}}.
\newblock Springer-Verlag, second edition, 1993.
\newblock Nonstiff problems.

\bibitem[HW96]{hw}
E.~Hairer and G.~Wanner.
\newblock {\em Solving ordinary differential equations. {II}}.
\newblock Springer-Verlag, second edition, 1996.
\newblock Stiff and differential-algebraic problems.

\bibitem[Kre68]{kreiss1}
H.-O. Kreiss.
\newblock Stability theory for difference approximations of mixed initial
  boundary value problems. {I}.
\newblock {\em Math. Comp.}, 22:703--714, 1968.

\bibitem[KW93]{kreiss-wu}
H.-O. Kreiss and L.~Wu.
\newblock On the stability definition of difference approximations for the
  initial-boundary value problem.
\newblock {\em Appl. Numer. Math.}, 12(1-3):213--227, 1993.

\bibitem[Ler53]{leray}
J.~Leray.
\newblock {\em Hyperbolic differential equations}.
\newblock The Institute for Advanced Study, Princeton, N. J., 1953.

\bibitem[Mic83]{michelson}
D.~Michelson.
\newblock Stability theory of difference approximations for multidimensional
  initial-boundary value problems.
\newblock {\em Math. Comp.}, 40(161):1--45, 1983.

\bibitem[Oli74]{oliger}
J.~Oliger.
\newblock Fourth order difference methods for the initial boundary-value
  problem for hyperbolic equations.
\newblock {\em Math. Comp.}, 28:15--25, 1974.

\bibitem[Rau72]{rauch}
J.~Rauch.
\newblock {${\mathcal L}^2$} is a continuable initial condition for {K}reiss'
  mixed problems.
\newblock {\em Comm. Pure Appl. Math.}, 25:265--285, 1972.

\bibitem[RM67]{RM}
R.~D. Richtmyer and K.~W. Morton.
\newblock {\em Difference methods for initial value problems}.
\newblock Graduate Texts in Mathematics. Interscience Publishers John Wiley \&
  Sons, 1967.
\newblock Theory and applications.

\bibitem[Slo83]{sloan}
D.~M. Sloan.
\newblock Boundary conditions for a fourth order hyperbolic difference scheme.
\newblock {\em Math. Comp.}, 41:1--11, 1983.

\bibitem[SW97]{strikwerda-wade}
J.~C. Strikwerda and B.~A. Wade.
\newblock A survey of the {K}reiss matrix theorem for power bounded families of
  matrices and its extensions.
\newblock In {\em Linear operators ({W}arsaw, 1994)}, volume~38 of {\em Banach
  Center Publ.}, pages 339--360. Polish Acad. Sci., 1997.

\bibitem[TE05]{TE}
L.~N. Trefethen and M.~Embree.
\newblock {\em Spectra and pseudospectra}.
\newblock Princeton University Press, 2005.
\newblock The behavior of nonnormal matrices and operators.

\bibitem[Tho72]{thomas}
J.~M. Thomas.
\newblock Discr\'etisation des conditions aux limites dans les sch\'emas
  saute-mouton.
\newblock {\em Rev. Fran\c caise Automat. Informat. Recherche
  Op\'erationnelle}, 6(Ser. R-2):31--44, 1972.

\bibitem[Tre84]{trefethen3}
L.~N. Trefethen.
\newblock Instability of difference models for hyperbolic initial boundary
  value problems.
\newblock {\em Comm. Pure Appl. Math.}, 37:329--367, 1984.

\bibitem[Wad90]{wade}
B.~A. Wade.
\newblock Symmetrizable finite difference operators.
\newblock {\em Math. Comp.}, 54(190):525--543, 1990.

\bibitem[Wu95]{wu}
L.~Wu.
\newblock The semigroup stability of the difference approximations for
  initial-boundary value problems.
\newblock {\em Math. Comp.}, 64(209):71--88, 1995.

\end{thebibliography}
\end{document}